\documentclass[runningheads,envcountsame]{llncs}
\usepackage{amsmath}
\newtheorem{Claim}{Claim}
\spnewtheorem*{proposition*}{Proposition}{\bfseries\scshape}{\itshape}
\spnewtheorem*{theorem*}{Theorem}{\bfseries\scshape}{\itshape}
\spnewtheorem*{lemma*}{Lemma}{\bfseries\scshape}{\itshape}
\usepackage{enumitem}
\usepackage[small]{caption}
\usepackage[subrefformat=parens]{subcaption}
\usepackage{tikz}
\usetikzlibrary{calc}
\usetikzlibrary{arrows}
\usepackage[hyphens]{url} 
\usepackage[hidelinks]{hyperref}

\AtBeginDocument{%
  \def\doi#1{\url{https://doi.org/#1}}}
\makeatother

\makeatletter
\renewcommand{\@Opargbegintheorem}[4]{%
  #4\trivlist\item[\hskip\labelsep{#3#2\@thmcounterend}]}
\makeatother

\newcommand{\bigraph}{bipartite graph}
\newcommand{\twoDORG}{two-directional orthogonal ray graph}
\newcommand{\TwoDORG}{Two-directional orthogonal ray graph}
\newcommand{\NR}{normalized representation}
\newcommand{\wrt}{with respect to}

\begin{document}
\title{A characterization of uniquely representable two-directional orthogonal ray graphs}
\titlerunning{On uniquely representable two-directional orthogonal ray graphs}
%
\author{Asahi~Takaoka}
\authorrunning{A.~Takaoka}
%
\institute{
  College of Information and Systems, 
  Muroran Institute of Technology, \\
  Mizumoto 27-1, Muroran, 
  Hokkaido, 050--8585, Japan \\
\email{takaoka@muroran-it.ac.jp}
}
\maketitle              
\begin{abstract}

In this paper, we provide a characterization of uniquely representable two-directional orthogonal ray graphs, which are defined as the intersection graphs of rightward and downward rays. 
The collection of these rays is called a representation of the graph.
Two-directional orthogonal ray graphs are equivalent to several well-studied classes of graphs, including complements of circular-arc graphs with clique cover number two. 
Normalized representations of two-directional orthogonal ray graphs, where the positions of certain rays are determined by neighborhood containment relations, can be obtained from the normalized representations of circular-arc graphs.
However, the normalized representations are not necessarily unique, even when considering only the relative positions of the rays.
Recent studies indicate that two-directional orthogonal ray graphs share similar characterizations to interval graphs.
Hanlon (1982) and Fishburn (1985) characterized uniquely representable interval graphs by introducing the notion of a buried subgraph.
Following their characterization, we define buried subgraphs of two-directional orthogonal ray graphs and prove that their absence is a necessary and sufficient condition for a graph to be uniquely representable.

\end{abstract}
\section{Introduction}
{\TwoDORG}s were originally introduced to facilitate 
the defect-tolerant design of nano-circuits~\cite{STU09-ISCAS,STU10-ISCAS,STTU11-IEICE},
and they are equivalent to 
the following important classes of graphs~\cite{STU10-DAM,TTU14-IEICE}: 
\begin{itemize}
\item 
complements of circular-arc graphs with clique cover number two, 
which provide a dichotomy for the list homomorphism problem~\cite{FHH99-Comb}
and play an important role in the recognition of 
circular-arc graphs~\cite{Spinrad88-JCTSB,HH97-GaC};
\item 
bipartite-complements of 2-chain subgraph coverable graphs,
which yield the fastest known algorithm 
for recognizing trapezoid graphs~\cite{MS94-JAL};
\item 
{\bigraph}s with interval dimension at most two, which 
admit a characterization in terms of forbidden induced subgraphs~\cite{TM76-DM}; and 
\item 
interval containment bigraphs and permutation bigraphs, 
which can be seen as a natural bipartite analogue of 
permutation graphs~\cite{Huang06-JCTSB,SBSW14-DAM}. 
\end{itemize}
Characterizations in terms of 
vertex orderings and forbidden submatrices were presented 
for {\twoDORG}s~\cite{Huang18-DAM,STU10-DAM}.
Permutation graphs and interval graphs are fundamental classes in graph theory~\cite{Golumbic04,BLS99}.
Recent studies suggest that 
{\twoDORG}s can be considered as a bipartite analogue of interval graphs 
because of their similar characterizations; 
for example, they share a common characterization called min ordering~\cite{HHMR20-SIDMA}, and 
the characterization that {\twoDORG}s are chordal bipartite and edge-asteroid-free resembles the characterization that interval graphs are chordal and asteroid-free~\cite{HHLM20-SIDMA}.

A {\bigraph} $G$ with bipartition $(U, V)$ is 
a \emph{{\twoDORG}}~\cite{STU10-DAM} if there exists 
a family of rightward rays (half-lines) $R(u)$, $u \in U$, 
in the $xy$-plane and 
a family of downward rays $R(v)$, $v \in V$, such that 
for any $u \in U$ and $v \in V$, $uv \in E(G)$ 
iff $R(u)$ intersects $R(v)$. 
The family of rays $R(G) = \{R(v) \colon\ v \in V(G)\}$ is 
called a \emph{representation} of $G$. 

This paper focuses on the relations between the graph structure and representations of {\twoDORG}s. 
Note that a representation yields only one graph, 
while a graph may have more than one representation.
In this paper, we characterize uniquely representable {\twoDORG}s. 
The relations between graphs and their representations have been widely investigated for interval graphs and permutation graphs. 
It was found that these graphs 
can be decomposed into a collection of components (i.e., indecomposable induced subgraphs) by modular decomposition, 
each of which has an essentially unique representation~\cite{Hsu95-SICOMP,MS94b-JAL},~\cite[page 191]{BLS99}.
This finding suggests the possibility that {\twoDORG}s can be decomposed into uniquely representable components by a decomposition scheme. 

Before discussing the characterization, we summarize the assumptions and basic facts when studying the representations of {\twoDORG}s.
\begin{enumerate}[label=\textup{(\arabic*)}]
\item 
Two vertices are called \emph{twins} if they have the same neighborhood. 
We assume that no {\twoDORG} contains twins in this paper, since twins can be assigned to the same ray.
\item 
We assume that the bipartition $(U, V)$ is given, and the vertices of $U$ and $V$ are assigned to rightward and downward rays, respectively. 
\item 
Two vertices $u \in U$ and $v \in V$ are adjacent if both the $x$- and $y$-coordinates of the endpoint of $R(u)$ are smaller than those of $R(v)$.
In other words, the adjacency among vertices is determined only by the endpoints of the corresponding rays. 
\item 
We focus only on the relative positions of the endpoints of the rays. Therefore, we dismiss, for example, the values of the $x$- and $y$-coordinates of the endpoints. Moreover, we assume that no endpoints have the same $x$- or $y$-coordinates.
\end{enumerate}
From above, we consider a representation as a pair of linear orders of the vertices, i.e., 
the increasing orders of the $x$- and $y$-coordinate values of the endpoints of the corresponding rays, 
and then we redefine {\twoDORG}s and their representations. 
\begin{definition}
A {\bigraph} $G$ with bipartition $(U, V)$ is a \emph{{\twoDORG}} 
if there is a pair of linear orders $(<_x, <_y)$ on $V(G)$ 
such that two vertices $u \in U$ and $v \in V$ 
are adjacent iff $u <_x v$ and $u <_y v$.
We refer to $(<_x, <_y)$ as a \emph{representation} of $G$. 
\end{definition}

A special type of representation, called {\NR}, is known for {\twoDORG}s, 
in which the relative positions among some vertices are fixed by neighborhood containment relations.
\begin{definition}\label{def:normalized representation}
Let $G$ be a {\twoDORG} with bipartition $(U, V)$. 
A representation $(<_x, <_y)$ of $G$ is \emph{normalized} 
if it satisfies the following conditions: 
\begin{enumerate}[label=\textup{(\alph*)}]
\item 
for any $u_1, u_2 \in U$, $u_1 <_x u_2 \wedge u_1 <_y u_2 \iff N(u_1) \supset N(u_2)$; 
\item 
for any $v_1, v_2 \in V$, $v_1 <_x v_2 \wedge v_1 <_y v_2 \iff N(v_1) \subset N(v_2)$; and 
\item 
for any $u \in U$ and $v \in V$, $v <_x u \wedge v <_y u \iff$ for any $v' \in N(u)$, $N(v) \subset N(v')$. 
\end{enumerate}
Note that the implications $\implies$ hold for any representation, 
whereas the other implications hold only for the {\NR}s.
\end{definition}

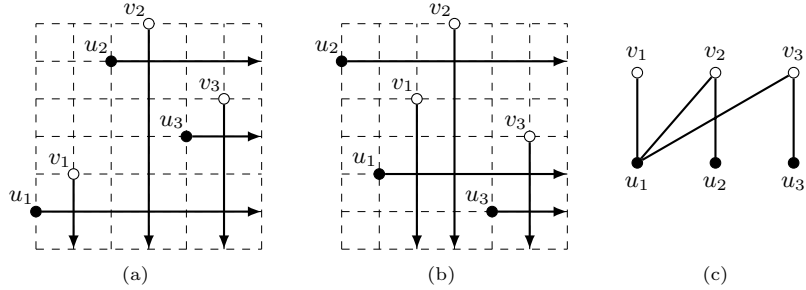
\begin{figure}[t]
\centering\subcaptionbox{\label{fig:E-representation}}{\begin{tikzpicture}[scale=0.5]
\useasboundingbox (-0.2, 0) rectangle (7.5, 7.0);

\foreach \x in {1, 2, 3, 4, 5, 6, 7}
    \draw[very thin,dashed] (\x, 0) -- (\x, 6);
\foreach \y in {0, 1, 2, 3, 4, 5, 6}
    \draw[very thin,dashed] (1, \y) -- (7, \y);

\tikzstyle{every node}=[draw,circle,fill=white,minimum size=4pt,inner sep=0pt]
\node [label=above left:$u_1$,fill=black] (u1) at (1, 1) {};
\node [label=above left:$u_2$,fill=black] (u2) at (3, 5) {};
\node [label=above left:$u_3$,fill=black] (u3) at (5, 3) {};
\node [label=above left:$v_1$] (v1) at (2, 2) {};
\node [label=above left:$v_2$] (v2) at (4, 6) {};
\node [label=above left:$v_3$] (v3) at (6, 4) {};
\draw [thick,-latex] (u1) -- (7, 1);
\draw [thick,-latex] (u2) -- (7, 5);
\draw [thick,-latex] (u3) -- (7, 3);
\draw [thick,-latex] (v1) -- (2, 0);
\draw [thick,-latex] (v2) -- (4, 0);
\draw [thick,-latex] (v3) -- (6, 0);
\end{tikzpicture}}
\centering\subcaptionbox{\label{fig:E-representation_2}}{\begin{tikzpicture}[scale=0.5]
\useasboundingbox (-0.2, 0) rectangle (7.5, 7.0);

\foreach \x in {1, 2, 3, 4, 5, 6, 7}
    \draw[very thin,dashed] (\x, 0) -- (\x, 6);
\foreach \y in {0, 1, 2, 3, 4, 5, 6}
    \draw[very thin,dashed] (1, \y) -- (7, \y);

\tikzstyle{every node}=[draw,circle,fill=white,minimum size=4pt,inner sep=0pt]
\node [label=above left:$u_1$,fill=black] (u1) at (2, 2) {};
\node [label=above left:$u_2$,fill=black] (u2) at (1, 5) {};
\node [label=above left:$u_3$,fill=black] (u3) at (5, 1) {};
\node [label=above left:$v_1$] (v1) at (3, 4) {};
\node [label=above left:$v_2$] (v2) at (4, 6) {};
\node [label=above left:$v_3$] (v3) at (6, 3) {};
\draw [thick,-latex] (u1) -- (7, 2);
\draw [thick,-latex] (u2) -- (7, 5);
\draw [thick,-latex] (u3) -- (7, 1);
\draw [thick,-latex] (v1) -- (3, 0);
\draw [thick,-latex] (v2) -- (4, 0);
\draw [thick,-latex] (v3) -- (6, 0);
\end{tikzpicture}}
\centering\subcaptionbox{\label{fig:E-graph}}{\begin{tikzpicture}
\useasboundingbox (-1.5, -1.75) rectangle (1.5, 1.75);

\tikzstyle{every node}=[draw,circle,fill=white,minimum size=4pt,inner sep=0pt]
\node [label=below:$u_1$,fill=black] (u1) at (210:1.2) {};
\node [label=below:$u_3$,fill=black] (u3) at (330:1.2) {};
\node [label=below:$u_2$,fill=black] (u2) at ($0.5*(u1)+0.5*(u3)$) {};
\node [label=above:$v_1$] (v1) at (150:1.2) {};
\node [label=above:$v_3$] (v3) at ( 30:1.2) {};
\node [label=above:$v_2$] (v2) at ($0.5*(v1)+0.5*(v3)$) {};

\draw [thick]
	(u1) -- (v1)
	(u1) -- (v2)
	(u1) -- (v3)
	(u2) -- (v2)
	(u3) -- (v3)
;
\end{tikzpicture}}
\caption{\subref{fig:E-representation} Normalized and \subref{fig:E-representation_2} non-normalized representations of a graph.}
\label{fig:E}
\end{figure}

\begin{example}
Figures~\ref{fig:E}\subref{fig:E-representation} and~\ref{fig:E}\subref{fig:E-representation_2} illustrate 
normalized and non-normalized representations 
of the graph in Fig.~\ref{fig:E}\subref{fig:E-graph}, respectively. 
In Fig.~\ref{fig:E}\subref{fig:E-representation}, 
$u_1 <_x u_2$ and $u_1 <_y u_2$ since $N(u_1) \supset N(u_2)$, 
whereas $u_1 >_x u_2$ in Fig.~\ref{fig:E}\subref{fig:E-representation_2}.
In Fig.~\ref{fig:E}\subref{fig:E-representation}, 
$v_1 <_x v_3$ and $v_1 <_y v_3$ since $N(v_1) \subset N(v_3)$, 
whereas $v_1 >_y v_3$ in Fig.~\ref{fig:E}\subref{fig:E-representation_2}.
Moreover, in Fig.~\ref{fig:E}\subref{fig:E-representation}, 
$v_1 <_x u_2$ and $v_1 <_y u_2$ 
since all vertices in $N(v_1)$ are adjacent to all of $N(u_2)$, 
whereas $v_1 >_x u_2$ in Fig.~\ref{fig:E}\subref{fig:E-representation_2}.
\end{example}

In the following, we will only address {\NR}s.
Note that, strictly speaking, {\NR}s are known 
for circular-arc graphs \cite{Spinrad88-JCTSB,Hsu95-SICOMP}. 
The {\NR}s provide efficient recognition and isomorphism testing algorithms 
for circular-arc graphs with clique cover number two~\cite{ES93-SODA,Eschen97-PhD}, 
which are precisely the complements of {\twoDORG}s~\cite{STU10-DAM}. 
However, to our knowledge, 
there is no literature addressing the {\NR}s of {\twoDORG}s.
In Sect.~\ref{sec:normalized representation}, 
we will formally prove that every graph has a {\NR}.

Although a {\NR} is a special type of representation, 
the following examples indicate that even a {\NR} is not necessarily unique. 
\begin{example}
In the representation in Fig.~\ref{fig:E}\subref{fig:E-representation}, 
the positions of $u_2$ and $v_2$ can be switched to those of $u_3$ and $v_3$, that is, 
the pair of linear orders 
$u_1 <_x v_1 <_x u_3 <_x v_3 <_x u_2 <_x v_2$ and 
$u_1 <_y v_1 <_y u_2 <_y v_2 <_y u_3 <_y v_3$ 
is a {\NR} of the graph.
Thus, the graph in Fig.~\ref{fig:E} has at least two {\NR}s. 
\end{example}

\begin{example}\label{ex:ext-sunlet4}
Figure~\ref{fig:ext-sunlet4} illustrates two {\NR}s of a graph. 
Again, the positions of $u_4$ and $v_4$ can be switched to those of $u_5$ and $v_5$. 
Thus, the graph in Fig.~\ref{fig:ext-sunlet4} has at least four {\NR}s.
\end{example}

\begin{figure}[t]
\centering\begin{tikzpicture}[scale=0.38]
\useasboundingbox (0, 0) rectangle (11, 11);

\foreach \x in {1, 2, 3, 4, 5, 6, 7, 8, 9, 10, 11}
    \draw[very thin,dashed] (\x, 0) -- (\x, 10);
\foreach \y in {0, 1, 2, 3, 4, 5, 6, 7, 8, 9, 10}
    \draw[very thin,dashed] (1, \y) -- (11, \y);

\tikzstyle{every node}=[draw,circle,fill=white,minimum size=4pt,inner sep=0pt]
\node [label=above left:$u_1$,fill=black] (u1) at (1, 3) {};
\node [label=above left:$u_2$,fill=black] (u2) at (3, 5) {};
\node [label=above left:$u_3$,fill=black] (u3) at (5, 1) {};
\node [label=above left:$u_4$,fill=black] (u4) at (7, 9) {};
\node [label=above left:$u_5$,fill=black] (u5) at (9, 7) {};
\node [label=above left:$v_1$] (v1) at (2, 4) {};
\node [label=above left:$v_2$] (v2) at (4, 6) {};
\node [label=above left:$v_3$] (v3) at (6, 2) {};
\node [label=above left:$v_4$] (v4) at (8, 10) {};
\node [label=above left:$v_5$] (v5) at (10, 8) {};

\draw [thick,-latex] (u1) -- (11, 3);
\draw [thick,-latex] (u2) -- (11, 5);
\draw [thick,-latex] (u3) -- (11, 1);
\draw [thick,-latex] (u4) -- (11, 9);
\draw [thick,-latex] (u5) -- (11, 7);
\draw [thick,-latex] (v1) -- (2, 0);
\draw [thick,-latex] (v2) -- (4, 0);
\draw [thick,-latex] (v3) -- (6, 0);
\draw [thick,-latex] (v4) -- (8, 0);
\draw [thick,-latex] (v5) -- (10, 0);

\end{tikzpicture}
\centering\begin{tikzpicture}[scale=0.38]
\useasboundingbox (0, 0) rectangle (11, 11);

\foreach \x in {1, 2, 3, 4, 5, 6, 7, 8, 9, 10, 11}
    \draw[very thin,dashed] (\x, 0) -- (\x, 10);
\foreach \y in {0, 1, 2, 3, 4, 5, 6, 7, 8, 9, 10}
    \draw[very thin,dashed] (1, \y) -- (11, \y);

\tikzstyle{every node}=[draw,circle,fill=white,minimum size=4pt,inner sep=0pt]
\node [label=above left:$u_1$,fill=black] (u1) at (3, 1) {};
\node [label=above left:$u_2$,fill=black] (u2) at (5, 3) {};
\node [label=above left:$u_3$,fill=black] (u3) at (1, 5) {};
\node [label=above left:$u_5$,fill=black] (u4) at (7, 9) {};
\node [label=above left:$u_4$,fill=black] (u5) at (9, 7) {};
\node [label=above left:$v_1$] (v1) at (4, 2) {};
\node [label=above left:$v_2$] (v2) at (6, 4) {};
\node [label=above left:$v_3$] (v3) at (2, 6) {};
\node [label=above left:$v_5$] (v4) at (8, 10) {};
\node [label=above left:$v_4$] (v5) at (10, 8) {};

\draw [thick,-latex] (u1) -- (11, 1);
\draw [thick,-latex] (u2) -- (11, 3);
\draw [thick,-latex] (u3) -- (11, 5);
\draw [thick,-latex] (u4) -- (11, 9);
\draw [thick,-latex] (u5) -- (11, 7);
\draw [thick,-latex] (v1) -- (4, 0);
\draw [thick,-latex] (v2) -- (6, 0);
\draw [thick,-latex] (v3) -- (2, 0);
\draw [thick,-latex] (v4) -- (8, 0);
\draw [thick,-latex] (v5) -- (10, 0);

\end{tikzpicture}
\centering\begin{tikzpicture}[scale=0.76]
\useasboundingbox (-2.2, -3.25) rectangle (2.2, 2.25);

\tikzstyle{every node}=[draw,circle,fill=white,minimum size=4pt,inner sep=0pt]
\node [label=below right:$u_1$,fill=black] (u1) at (270:2.0) {};
\node [label=above      :$u_2$,fill=black] (u2) at (  0:0.0) {};
\node [label=above right:$u_3$,fill=black] (u3) at ( 90:1.0) {};
\node [label=above      :$u_4$,fill=black] (u4) at (180:2.0) {};
\node [label=above      :$u_5$,fill=black] (u5) at (  0:2.0) {};
\node [label=      right:$v_1$] (v1) at (270:3.0) {};
\node [label=above right:$v_2$] (v2) at (270:1.0) {};
\node [label=      right:$v_3$] (v3) at ( 90:2.0) {};
\node [label=above      :$v_4$] (v4) at (180:1.0) {};
\node [label=above      :$v_5$] (v5) at (  0:1.0) {};

\draw [thick]
	(v1) -- (u1) -- (v2) -- (u2)
	(u3) -- (v3)
	(u4) -- (v4)
	(u5) -- (v5)
	(v4) -- (u1) -- (v5)
	(v4) -- (u2) -- (v5)
	(v4) -- (u3) -- (v5)
;
\end{tikzpicture}
\caption{Two {\NR}s of a graph.}
\label{fig:ext-sunlet4}
\end{figure}
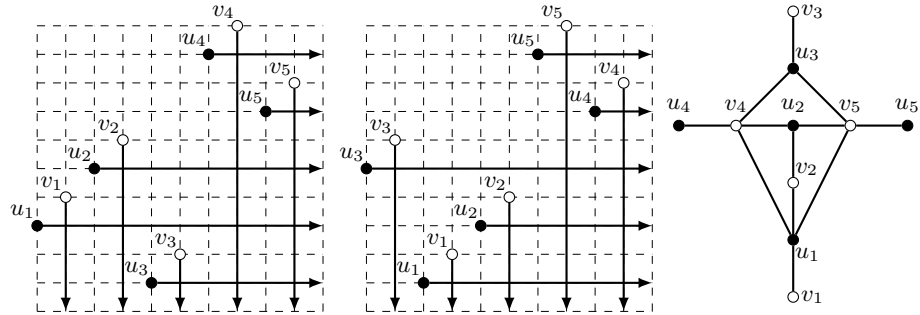

{\TwoDORG}s allowing only one {\NR} are chain graphs.
By definition, every chain graph has exactly one {\NR}. 
We will prove the converse in Sect.~\ref{sec:(ii) <==> (iii)}.
\begin{theorem}\label{thm:chain graph}
A {\twoDORG} has exactly one {\NR} iff it is a chain graph.
\end{theorem}

It is obvious that if $(<_x, <_y)$ is a {\NR} of a {\twoDORG} $G$, 
then the same is true for its reversed pair $(<_y, <_x)$; 
see Fig.~\ref{fig:ext-sunlet4} for an example.
Theorem~\ref{thm:chain graph} indicates that 
$(<_x, <_y)$ and $(<_y, <_x)$ differ unless $G$ is a chain graph.
Following this observation, we say that a graph is uniquely representable if it has exactly two {\NR}s. 
As shown in Example~\ref{ex:ext-sunlet4}, not all graphs are uniquely representable.
\begin{definition}
A {\twoDORG} is \emph{uniquely representable} 
if it has exactly two normalized representations. 
\end{definition}

Recall that {\twoDORG}s can be considered as a bipartite analogue of interval graphs.
An \emph{interval graph} is the intersection graph of a family of intervals on the real line.
The family of intervals is called a \emph{representation} of the graph. 
Interval graphs having essentially unique representation 
(i.e., unique up to a rearrangement of consecutive endpoints and a flip of the real line) 
were characterized by Hanlon~\cite{Hanlon82-TAMS} and Fishburn~\cite{Fishburn85-DAM}. 
Here, we use the characterization reformulated by Fiori-Carones and Marcone~\cite{FM22-DM}. 
\begin{definition}[\cite{Hanlon82-TAMS,Fishburn85-DAM,FM22-DM}]\label{def:buried subgraph of interval graphs}
Let $G$ be a graph. 
For a vertex set $B \subseteq V(G)$, let 
$K(B) = \{v \in V(G) \colon\ \text{$v$ is adjacent to all of $B$}\}$ and 
$R(B) = V(G) \setminus (B \cup K(B))$. 
A vertex set $B \subseteq V(G)$ is called a \emph{buried subgraph} if 
\begin{enumerate}[label=\textup{(\alph*)}]
\item there exist $a, b \in B$ such that $ab \notin E(G)$, 
\item $K(B) \cap B = \emptyset$ and $R(B) \neq \emptyset$, and 
\item if $b \in B$ and $r \in R(B)$, then $br \notin E(G)$. 
\end{enumerate}
Note that by definition, 
each vertex of $G - B$ is adjacent to all of $B$ or none of $B$, and 
the induced subgraph $G[B]$ contains no universal vertex.
\end{definition}

\begin{theorem}[\cite{Hanlon82-TAMS,Fishburn85-DAM,FM22-DM}]
A disconnected interval graph is uniquely representable 
iff it is the union of two complete graphs. 
\end{theorem}

\begin{theorem}[\cite{Hanlon82-TAMS,Fishburn85-DAM,FM22-DM}]
Let $G$ be a connected interval graph. 
Let $G^+$ be a graph such that 
$V(G^+) = \{(a, b) \in V(G)^2 \colon\ ab \notin E(G)\}$ 
and two vertices $(a, b)$ and $(c, d)$ are adjacent iff $ac, bd \in E(G)$. 
The following are equivalent:
\begin{enumerate}[label=\textup{(\roman*)}]
\item $G$ is uniquely representable;
\item $G$ contains no buried subgraph; and
\item the graph $G^+$ has two non-trivial components.
\end{enumerate}
\end{theorem}

\subsubsection{Our Results.}
In this paper, we define the buried subgraph of a {\twoDORG} 
and provide analogous results.
\begin{definition}
Let $G$ be a {\bigraph} with bipartition $(U, V)$. 
Two edges $u_1v_1$ and $u_2v_2$ of $G$ are said to be \emph{independent} 
if they are disjoint and neither $u_1v_2$ nor $u_2v_1$ is an edge of $G$. 
\end{definition}

\begin{definition}\label{def:buried subgraph}
Let $G$ be a {\bigraph} with bipartition $(U, V)$. 
We say that a vertex set $B \subseteq V(G)$ is a \emph{buried subgraph} if 
\begin{enumerate}[label=\textup{(\alph*)}]
\item there are independent edges in $G[B]$,
\item 
each vertex of $U \setminus B$ is adjacent to either all of $B \cap V$ or none of them, and 
each vertex of $V \setminus B$ is adjacent to either all of $B \cap U$ or none of them,
\item there is at least one edge in $G - B$ independent with another edge in $G$, and
\item no vertex of $B$ is isolated or universal in $G[B]$. 
\end{enumerate}
\end{definition}

\begin{example}
The graph in Fig.~\ref{fig:ext-sunlet4} contains two buried subgraphs 
$\{u_1, v_1, u_2, v_2,$ $u_3, v_3\}$ and $\{u_4, v_4, u_5, v_5\}$. 
\end{example}

\begin{definition}\label{def:G^+}
Let $G$ be a {\bigraph} with bipartition $(U, V)$. 
The \emph{auxiliary graph $G^+$} of $G$ is a graph such that 
$V(G^+) = \{(a, b) \in (U \times V) \cup (V \times U) \colon\ ab \notin E(G)\}$ 
and two vertices $(a, b)$ and $(c, d)$ are adjacent iff $ac, bd \in E(G)$; 
see Fig.~\ref{fig:G^+}. 
\end{definition}

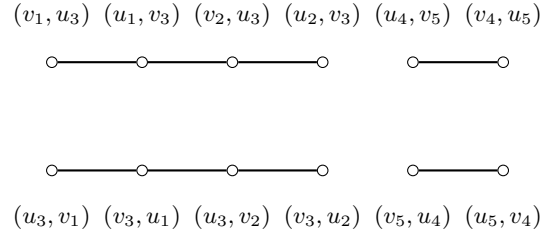
\begin{figure}[t]
\centering\begin{tikzpicture}[scale=1.2]
\useasboundingbox (-0.5, -0.8) rectangle (5.5, 1.2+0.8);

\tikzstyle{every node}=[draw,circle,fill=white,minimum size=4pt,inner sep=0pt]
\node [label=above:{\small $(v_1,u_3)$}] (v1u3) at (0, 1.2) {};
\node [label=above:{\small $(u_1,v_3)$}] (u1v3) at (1, 1.2) {};
\node [label=above:{\small $(v_2,u_3)$}] (v2u3) at (2, 1.2) {};
\node [label=above:{\small $(u_2,v_3)$}] (u2v3) at (3, 1.2) {};
\node [label=below:{\small $(u_3,v_1)$}] (u3v1) at (0, 0) {};
\node [label=below:{\small $(v_3,u_1)$}] (v3u1) at (1, 0) {};
\node [label=below:{\small $(u_3,v_2)$}] (u3v2) at (2, 0) {};
\node [label=below:{\small $(v_3,u_2)$}] (v3u2) at (3, 0) {};
\draw [thick] 
	(v1u3) -- (u1v3) -- (v2u3) -- (u2v3)
	(u3v1) -- (v3u1) -- (u3v2) -- (v3u2)
;
\node [label=above:{\small $(u_4,v_5)$}] (u4v5) at (4, 1.2) {};
\node [label=above:{\small $(v_4,u_5)$}] (v4u5) at (5, 1.2) {};
\node [label=below:{\small $(v_5,u_4)$}] (v5u4) at (4, 0) {};
\node [label=below:{\small $(u_5,v_4)$}] (u5v4) at (5, 0) {};
\draw [thick] 
	(u4v5) -- (v4u5)
	(u5v4) -- (v5u4)
;
\end{tikzpicture}
\caption{
	The auxiliary graph $G^+$ 
	of the graph $G$ in Fig.~\ref{fig:ext-sunlet4}.
	Trivial components are omitted.
}
\label{fig:G^+}
\end{figure}

\begin{theorem}\label{thm:disconnected}
A disconnected {\twoDORG} is uniquely representable iff 
it has two non-trivial components 
each of which is a chain graph. 
\end{theorem}

\begin{theorem}\label{thm:main}
The following are equivalent for a connected {\twoDORG} $G$:
\begin{enumerate}[label=\textup{(\roman*)}]
\item $G$ is uniquely representable;
\item $G$ contains no buried subgraph; and
\item the graph $G^+$ has two non-trivial components.
\end{enumerate}
\end{theorem}

\subsubsection{Organization of the Paper.}
The paper is organized as follows. 
In Sect.~\ref{sec:preliminaries}, 
we provide the basic definitions and notations used in this paper. 
In Sect.~\ref{sec:normalized representation}, 
we prove that every {\twoDORG} has a {\NR}. 
In Sect.~\ref{sec:normalized weak ordering}, 
we introduce a characterization called normalized weak ordering. 
In Sect.~\ref{sec:proof}, we prove Theorem~\ref{thm:main}.
The equivalence (i) $\iff$ (iii) is proven in Sect.~\ref{sec:(i) <==> (iii)} by using correspondences among the {\NR}s of a graph $G$, non-trivial components of $G^+$, and transitive orientations of the independence graph $I(G)$ of $G$. 
The equivalence (ii) $\iff$ (iii) is proven in Sect.~\ref{sec:(ii) <==> (iii)} by using normalized weak ordering. 
In addition, we prove Theorems~\ref{thm:chain graph} and~\ref{thm:disconnected} in Sect.~\ref{sec:(i) <==> (iii)}. 
We provide some remarks and discuss future research directions in Sect.~\ref{sec:conclusion}.

\section{Preliminaries}\label{sec:preliminaries}

We assume all graphs in this paper are finite and simple. 
Unless stated otherwise, graphs are assumed to be undirected, 
but we also deal with directed graphs. 
We write $uv$ for the (undirected) edge between $u$ and $v$, and 
$(u, v)$ for the directed edge from $u$ to $v$.
We sometimes write 
$u \sim_G v$ to denote that $u$ and $v$ are adjacent in $G$ and 
$u \to_G v$ to denote that there is a directed edge from $u$ to $v$ in $G$. 
We will omit the index $G$ if it is clear from the context.
For two sets $A$ and $B$, we write 
$A \subseteq B$ to denote that $A$ is a subset of $B$ and 
$A \subset B$ to denote that $A$ is a proper subset of $B$.

For a graph $G$, we write $V(G)$ and $E(G)$ 
for the vertex set and the edge set of $G$, respectively. 
The \emph{complement} of $G$ is the graph $\overline{G}$ 
such that $V(\overline{G}) = V(G)$ and 
$uv \in E(\overline{G}) \iff uv \notin E(G)$ 
for any $u, v \in V(\overline{G})$. 
For a subset of vertices $S \subseteq V(G)$, 
the graph \emph{induced} by $S$, denoted by $G[S]$, is the graph $H$ 
such that 
$V(H) = S$ and $uv \in E(H) \iff uv \in E(G)$ for any $u, v \in V(H)$. 
A graph is called an \emph{induced subgraph} of $G$ 
if it is induced by some set of vertices. 
We write $G - S$ for the graph obtained from $G$ by removing every vertex in $S$, that is, $G[V(G) \setminus S]$.
A subset of vertices is called 
a \emph{clique} if it induces a complete graph and 
an \emph{independent set} if no two of which are adjacent.
A vertex $v$ of $G$ is said to be 
\emph{isolated} if it is not adjacent to any other vertex in $G$ and 
\emph{universal} if it is adjacent to all other vertices in $G$. 
The (open) \emph{neighborhood} of $v$ is the set $N_G(v) = \{u \in V(G) \colon\ uv \in E(G)\}$ and 
the \emph{closed neighborhood} of $v$ is $N_G[v] = N_G(v) \cup \{v\}$.
We will omit the index $G$ if it is clear from the context. 
Two vertices $u$ and $v$ are called \emph{twins} if $N(u) = N(v)$.

A graph $G$ is \emph{bipartite} 
if $V(G)$ can be partitioned into two independent sets $U$ and $V$.
The pair $(U, V)$ is called \emph{bipartition} of $G$. 
In a {\bigraph}, we call a vertex \emph{universal} 
if it is adjacent to all vertices on the other side of the bipartition.
We denote every vertex in $U$ by the letter $u$ and 
every vertex in $V$ by the letter $v$. 
A subset of vertices is called 
a \emph{biclique} if it induces a complete {\bigraph}. 
A {\bigraph} $G$ with bipartition $(U, V)$ is a \emph{chain graph} 
if there is a linear order $<_U$ on $U$ such that 
$u_1 <_U u_2 \iff N(u_1) \subseteq N(u_2)$ for any $u_1, u_2 \in U$.
Equivalently, a {\bigraph} is a chain graph 
if there is a linear order $<_V$ on $V$ such that 
$v_1 <_V v_2 \iff N(v_1) \subseteq N(v_2)$ for any $v_1, v_2 \in V$.
Chain graphs are known to be {\bigraph}s 
containing no independent edges~\cite{Yannakakis82-SIADM}.
A {\bigraph} is said to be \emph{chordal bipartite} if it does not contain 
any cycle of length at least 6 as an induced subgraph. 
{\TwoDORG}s are known to be chordal bipartite~\cite{STU10-DAM}. 

For a graph $G$, an \emph{orientation} of $G$ is a directed graph $F$ 
obtained from $G$ by replacing each edge $uv \in E(G)$ with either $(u, v)$ or $(v, u)$. 
An orientation $F$ is \emph{transitive} 
if $u \to_F v \to_F w$ implies $u \to_F w$ for any $u, v, w \in V(F)$.
A graph admitting a transitive orientation is called 
a \emph{comparability graph} or a \emph{transitively orientable graph}.

\subsection{Normalized Representations}\label{sec:normalized representation}
In this section, we prove that 
every {\twoDORG} has a {\NR}. 
This is almost evident from~\cite{Spinrad88-JCTSB}, 
but to our knowledge, 
there is no literature explicitly stating it.

A {\bigraph} is a {\twoDORG} iff its complement is a circular-arc graph~\cite{STU10-DAM}.
A graph $G$ is a \emph{circular-arc graph} 
if for each vertex $v \in V(G)$, there is a circular arc $A(v)$ on a circle 
such that for any $u, v \in V(G)$, 
$uv \in E(G)$ iff $A(u)$ intersects $A(v)$. 
The set of arcs $A(G) = \{A(v) \colon\ v \in V(G)\}$ is called a \emph{representation} of $G$. 
A graph is said to be \emph{clique cover number two} if its vertex set can be partitioned into two cliques, 
i.e., it is the complement of a {\bigraph}. 

Spinrad~\cite{Spinrad88-JCTSB} showed that 
every circular-arc graph with clique cover number two has the following representation. 
Following Hsu~\cite{Hsu95-SICOMP}, 
we refer to this representation as a {\NR}.
\begin{theorem}[\cite{Spinrad88-JCTSB}]
Let $G$ be a circular-arc graph whose vertices can be partitioned into two cliques $U$ and $V$. 
The graph admits a \emph{{\NR}} $A(G) = \{A(v) \colon\ v \in V(G)\}$ 
that satisfies the following:
\begin{itemize}
\item the endpoints of the circular arcs are distinct;
\item there are two points $p$ and $q$ on the circle such that 
every arc $A(u)$, $u \in U$, contains $p$ but not $q$, and 
every arc $A(v)$, $v \in V$, contains $q$ but not $p$;
\item for any $a, b \in V(G)$, $A(a) \subset A(b) \iff N[a] \subset N[b]$;
\item for any $u \in U$ and $v \in V$, 
$A(u)$ and $A(v)$ cover the circle iff 
$uv \in E(G)$ and $N[v'] \subset N[v]$ for any $v' \in V(G) \setminus N[u]$.
\end{itemize}
\end{theorem}

\begin{example}
Figure~\ref{fig:E-bar} illustrates a {\NR} of a circular-arc graph.
In this representation, 
$A(u_1) \subset A(u_2)$ and $A(u_1) \subset A(u_3)$ while 
$N[u_1] \subset N[u_2]$ and $N[u_1] \subset N[u_3]$. 
On the other hand, $A(u_2)$ and $A(u_3)$ strictly overlap while 
$N[u_2] \not\subset N[u_3]$ and $N[u_2] \not\supset N[u_3]$. 
Moreover, $A(v_1)$ and $A(u_2)$ cover the circle while 
$N[v] \subset N[v_1]$ for any $v \in V(G) \setminus N[u_2]$. 
\end{example}

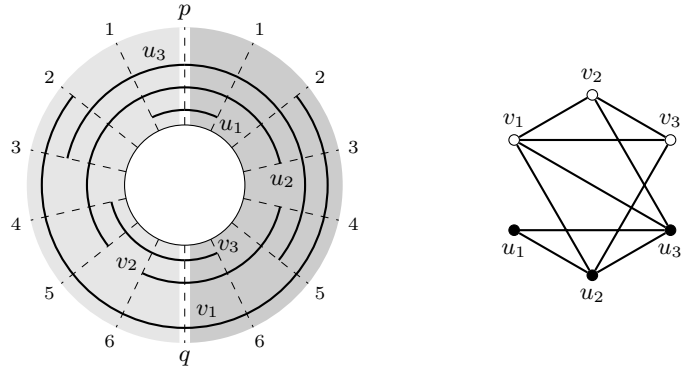
\begin{figure}[t]
\centering\begin{tikzpicture}
\useasboundingbox (-3.2, -2.5) rectangle (3.2, 2.5);

\filldraw[fill=black!10!white,draw=none] (90:0.8) arc (90:270:0.8) -- (270:2.1) arc (270:90:2.1) -- cycle;
\filldraw[fill=black!20!white,draw=none] (90:0.8) arc (90:-90:0.8) -- (-90:2.1) arc (-90:90:2.1) -- cycle;
\draw [line width = 4pt,draw=white] (90:2.1) -- (-90:2.1);
\draw [] (0, 0) circle [radius=0.8];
\draw [very thin,dashed] (90:  0.8) -- (90:  2.1); 
\draw [very thin,dashed] (270: 0.8) -- (270: 2.1); 
\draw (90:  2.3) node [anchor=center] {$p$};
\draw (270: 2.3) node [anchor=center] {$q$};
\foreach \x in {1, 2, 3, 4, 5, 6} {
    \draw [very thin,dashed] (90-\x*180/7: 0.8) -- (90-\x*180/7: 2.1); 
    \draw (90-\x*180/7: 2.3) node [anchor=center] {{\scriptsize $\x$}};
}
\foreach \x in {1, 2, 3, 4, 5, 6} {
    \draw [very thin,dashed] (90+\x*180/7: 0.8) -- (90+\x*180/7: 2.1); 
    \draw (90+\x*180/7: 2.3) node [anchor=center] {{\scriptsize $\x$}};
}

\draw [thick] (90:1.0) arc (90:90-1*180/7:1.0);
\draw [thick] (90:1.0) arc (90:90+1*180/7:1.0);
\draw [thick] (90:1.3) arc (90:90-3*180/7:1.3);
\draw [thick] (90:1.3) arc (90:90+5*180/7:1.3);
\draw [thick] (90:1.6) arc (90:90-5*180/7:1.6);
\draw [thick] (90:1.6) arc (90:90+3*180/7:1.6);

\draw (90-1.5*180/7: 1.0) node {$u_1$};
\draw (90-3.4*180/7: 1.3) node {$u_2$};
\draw (90+0.5*180/7: 1.8) node {$u_3$};

\draw [thick] (270:1.0) arc (270:270-3*180/7:1.0);
\draw [thick] (270:1.0) arc (270:270+1*180/7:1.0);
\draw [thick] (270:1.3) arc (270:270-1*180/7:1.3);
\draw [thick] (270:1.3) arc (270:270+3*180/7:1.3);
\draw [thick] (270:1.9) arc (270:270-5*180/7:1.9);
\draw [thick] (270:1.9) arc (270:270+5*180/7:1.9);

\draw (270+0.4*180/7: 1.7) node {$v_1$};
\draw (270-1.4*180/7: 1.3) node {$v_2$};
\draw (270+1.4*180/7: 1.0) node {$v_3$};

\end{tikzpicture}
\centering\begin{tikzpicture}
\useasboundingbox (-2.0, -2.5) rectangle (2.0, 2.5);

\tikzstyle{every node}=[draw,circle,fill=white,minimum size=4pt,inner sep=0pt]
\node [label=below:$u_1$,fill=black] (u1) at (210:1.2) {};
\node [label=below:$u_2$,fill=black] (u2) at (270:1.2) {};
\node [label=below:$u_3$,fill=black] (u3) at (330:1.2) {};
\node [label=above:$v_1$] (v1) at (150:1.2) {};
\node [label=above:$v_2$] (v2) at ( 90:1.2) {};
\node [label=above:$v_3$] (v3) at ( 30:1.2) {};

\draw [thick]
	(u1) -- (u2) -- (u3) -- (u1)
	(v1) -- (v2) -- (v3) -- (v1)
	(u2) -- (v3)
	(u3) -- (v2)
	(u2) -- (v1)
	(u3) -- (v1)
;
\end{tikzpicture}
\caption{A {\NR} of a circular-arc graph with clique cover number two.}
\label{fig:E-bar}
\end{figure}

The {\NR}s of {\twoDORG}s can be obtained from those of circular-arc graphs.
\begin{theorem}
Every {\twoDORG} has a normalized representation.
\end{theorem}
\begin{proof}
Let $G$ be a {\twoDORG} with bipartition $(U, V)$.
The complement $\overline{G}$ of $G$ is a circular-arc graph
whose vertices can be partitioned into two cliques $U$ and $V$~\cite{STU10-DAM}.
Then, $\overline{G}$ has a {\NR} $A(\overline{G})$.
We construct a {\NR} of $G$ from $A(\overline{G})$.
When moving around the circle from $p$ to $q$ in the clockwise direction, 
we obtain a linear order of the vertices $<_1$ 
from the order of the occurrence of the endpoints of the arcs.
Since the endpoints are distinct, $<_1$ is well defined.
Similarly, 
when moving around the circle from $p$ to $q$ in the counterclockwise direction, 
we obtain a linear order of the vertices $<_2$ 
from the order of the occurrence of the endpoints.
It is routine to check that $(<_1, <_2)$ is a normalized representation of $G$. 
\end{proof}

\begin{example}
Regarding the representation in Fig.~\ref{fig:E-bar}, we obtain 
$u_1 <_1 v_1 <_1 u_2 <_1 v_2 <_1 u_3 <_1 v_3$ from the right half of the circle 
and 
$u_1 <_2 v_1 <_2 u_3 <_2 v_3 <_2 u_2 <_2 v_2$ from the left half of the circle. 
The pair $(<_1, <_2)$ is the {\NR} in Fig.~\ref{fig:E}\subref{fig:E-representation}. 
\end{example}

\subsection{Normalized Weak Ordering}\label{sec:normalized weak ordering}

{\TwoDORG}s admits a characterization called weak ordering~\cite{STU10-DAM}. 
In this section, 
we introduce a slightly modified version of this characterization. 
The new characterization is directly derived from the {\NR} and 
will be used in Sect.~\ref{sec:(ii) <==> (iii)}. 

\begin{definition}[\cite{STU10-DAM}]
Let $G$ be a {\bigraph} with bipartition $(U, V)$.
A pair of linear orders $<_U$, $<_V$ of $U, V$, respectively, 
is called a \emph{weak ordering} if 
for every $u_1, u_2 \in U$ and $v_1, v_2 \in V$ 
with $u_1 <_U u_2$ and $v_1 <_V v_2$, 
$u_1v_2 \in E(G)$ and $u_2v_1 \in E(G)$ imply $u_2v_2 \in E(G)$. 
\end{definition}

\begin{theorem}[\cite{STU10-DAM}]
A {\bigraph} is a {\twoDORG} iff it has a weak ordering.
\end{theorem}

Now, we introduce the normalized version of weak ordering, 
in which the order among some vertices is determined by neighborhood containment relations. 

\begin{definition}
Let $G$ be a {\twoDORG} with bipartition $(U, V)$.
We say that its weak ordering $(<_U, <_V)$ is \emph{normalized} if 
\begin{itemize}
\item for every $u_1, u_2 \in U$, $N(u_1) \subset N(u_2)$ implies $u_1 <_U u_2$, and 
\item for every $v_1, v_2 \in V$, $N(v_1) \subset N(v_2)$ implies $v_1 <_V v_2$. 
\end{itemize}
\end{definition}

\begin{theorem}
Every {\twoDORG} has a normalized weak ordering.
\end{theorem}
\begin{proof}
Let $G$ be a {\twoDORG} with bipartition $(U, V)$, 
and let $(<_x, <_y)$ be its {\NR}.
We define linear orders $<_U$, $<_V$ of $U, V$, respectively, so that 
\[u_1 <_U u_2 \iff u_1 >_y u_2 \text{ and } v_1 <_V v_2 \iff v_1 <_x v_2\]
for every $u_1, u_2 \in U$ and $v_1, v_2 \in V$. 
It is straightforward to see that
$(<_U, <_V)$ is a weak ordering of $G$. 
By the property of {\NR}, 
for every $u_1, u_2 \in U$, 
$N(u_1) \subset N(u_2) \implies u_1 >_y u_2 \iff u_1 <_U u_2$. 
Similarly, 
for every $v_1, v_2 \in V$, 
$N(v_1) \subset N(v_2) \implies v_1 <_x v_2 \iff v_1 <_V v_2$. 
Thus, $(<_U, <_V)$ is normalized.
\end{proof}

\begin{example}
From the representation in Fig.~\ref{fig:E}\subref{fig:E-representation},
we obtain linear orders 
$u_2 <_U u_3 <_U u_1$ and $v_1 <_V v_2 <_V v_3$ of 
$U$ and $V$, respectively. 
We can easily check that $(<_U, <_V)$ 
is a normalized weak ordering.
\end{example}

\section{Proof of Theorem~\ref{thm:main}}\label{sec:proof}

\subsection{Proof of (i) $\iff$ (iii)}\label{sec:(i) <==> (iii)}

In this section, we prove the equivalence (i) $\iff$ (iii) of Theorem~\ref{thm:main}.
First, we discuss the properties of the auxiliary graph $G^+$ in Definition~\ref{def:G^+}. 
The following is a key observation.
\begin{proposition}\label{prop:independent edges}
Let $G$ be a {\twoDORG} with bipartition $(U, V)$, and 
let $u_1v_1$ and $u_2v_2$ be independent edges of $G$. 
For any representation $(<_x, <_y)$ of $G$, 
\[u_1 <_x v_2 \iff v_2 <_y u_1 \iff u_2 <_y v_1 \iff v_1 <_x u_2.\]
\end{proposition}
\begin{proof}
Suppose $u_1 <_x v_2$. 
Then, $v_2 <_y u_1$; otherwise, $u_1v_2 \in E(G)$. 
As $u_1v_1,$ $u_2v_2 \in E(G)$ implies $u_1 <_y v_1$ and $u_2 <_y v_2$, 
we have $u_2 <_y v_1$. 
Now, $v_1 <_x u_2$; otherwise, $u_2v_1 \in E(G)$. 
As $u_1v_1, u_2v_2 \in E(G)$ implies $u_1 <_x v_1$ and $u_2 <_x v_2$, 
we have $u_1 <_x v_2$. 
\end{proof}

Consider independent edges $u_1v_1$ and $u_2v_2$ in a {\twoDORG} $G$. 
Proposition~\ref{prop:independent edges} indicates that 
\[
\left\{
\begin{aligned}
& u_1 <_x v_1 <_x u_2 <_x v_2 \\
& u_2 <_y v_2 <_y u_1 <_y v_1 
\end{aligned}
\right.
\text{ or }
\left\{
\begin{aligned}
& u_2 <_x v_2 <_x u_1 <_x v_1 \\
& u_1 <_y v_1 <_y u_2 <_y v_2
\end{aligned}
\right.
\]
for any representation $(<_x, <_y)$ of $G$. 
The auxiliary graph $G^+$ represents this forcing relations 
and provides a necessary condition 
for a graph to be a {\twoDORG}. 
\begin{definition}
Let $G$ be a {\bigraph}. 
An \emph{invertible pair} of $G$ is 
a pair of vertices $u, v$ such that 
in $G^+$, two vertices $(u, v)$ and $(v, u)$ are 
in the same connected component.
\end{definition}

It is straightforward to see that 
no {\twoDORG} contains an invertible pair. 
The importance of $G^+$ is that the converse also holds. 
\begin{theorem}\label{thm:invertible pair}
A {\bigraph} is a {\twoDORG} iff it contains no invertible pair.
\end{theorem}
\begin{proof}
This proof is simply a reformulation of the theorem of Hell and Huang \cite{HH97-GaC,Huang18-incollection}. 
They showed that a {\bigraph} $G$ is 
the complement of a circular-arc graph 
(i.e., $G$ is a {\twoDORG}) iff 
the graph $G^*$ defined as follows is bipartite: 
$V(G^*) = \{(a, b) \in (U \times V) \cup (V \times U) \colon\ ab \notin E(G)\}$ 
and each vertex $(a, b)$ is adjacent to $(b, a)$ and $(c, d)$ if $ad, bc \in E(G)$; see Fig.~\ref{fig:G^*}. 
It is easy to see that 
$G$ contains an invertible pair iff 
$G^*$ contains an odd cycle; 
see Figs.~\ref{fig:G^+} and~\ref{fig:G^*}.
\end{proof}

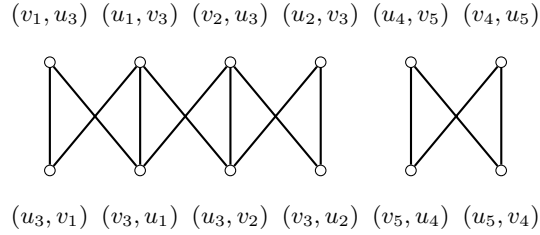
\begin{figure}[t]
\centering
\begin{tikzpicture}[scale=1.2]
\useasboundingbox (-0.5, -0.8) rectangle (5.5, 1.2+0.8);

\tikzstyle{every node}=[draw,circle,fill=white,minimum size=4pt,inner sep=0pt]
\node [label=above:{\small $(v_1,u_3)$}] (v1u3) at (0, 1.2) {};
\node [label=above:{\small $(u_1,v_3)$}] (u1v3) at (1, 1.2) {};
\node [label=above:{\small $(v_2,u_3)$}] (v2u3) at (2, 1.2) {};
\node [label=above:{\small $(u_2,v_3)$}] (u2v3) at (3, 1.2) {};
\node [label=below:{\small $(u_3,v_1)$}] (u3v1) at (0, 0) {};
\node [label=below:{\small $(v_3,u_1)$}] (v3u1) at (1, 0) {};
\node [label=below:{\small $(u_3,v_2)$}] (u3v2) at (2, 0) {};
\node [label=below:{\small $(v_3,u_2)$}] (v3u2) at (3, 0) {};
\draw [thick] 
	(v1u3) -- (u3v1) -- (u1v3) -- (v3u1) -- (v1u3)
	(u1v3) -- (u3v2) -- (v2u3) -- (v3u1)
	(v2u3) -- (v3u2) -- (u2v3) -- (u3v2)
;
\node [label=above:{\small $(u_4,v_5)$}] (u4v5) at (4, 1.2) {};
\node [label=above:{\small $(v_4,u_5)$}] (v4u5) at (5, 1.2) {};
\node [label=below:{\small $(v_5,u_4)$}] (v5u4) at (4, 0) {};
\node [label=below:{\small $(u_5,v_4)$}] (u5v4) at (5, 0) {};
\draw [thick] 
	(u4v5) -- (v5u4) -- (v4u5) -- (u5v4) -- (u4v5)
;
\end{tikzpicture}
\caption{
        The auxiliary graph $G^*$ of the graph $G$ in Fig.~\ref{fig:ext-sunlet4}. 
        Independent edges are omitted. 
}
\label{fig:G^*}
\end{figure}

Theorem~\ref{thm:invertible pair} indicates that 
if $G$ is a {\twoDORG}, 
no vertex $(u, v)$ in $G^+$ is contained in the same component as its reverse $(v, u)$. 
Here, we state it explicitly. 
\begin{definition}
Let $C$ be a connected component of $G^+$. 
The \emph{reversed component} of $C$, denoted by $C^{-1}$, 
is the component of $G^+$ such that 
$(u, v) \in C^{-1}$ for some $(v, u) \in C$. 
Note that, by definition, 
$(u, v) \in C^{-1} \iff (v, u) \in C$ 
for any $(u, v) \in V(G^+)$. 
\end{definition}
\begin{proposition}
A {\bigraph} $G$ is a {\twoDORG} iff 
each component of $G^+$ differs 
from its reversed component.
\end{proposition}

Next, to investigate the relations between 
{\NR}s of a graph $G$ and connected components of $G^+$, 
we introduce another auxiliary graph. 
\begin{definition}
Let $G$ be a {\bigraph}. 
The \emph{independence graph} $I(G)$ of $G$ is a graph whose vertices are the edges of $G$ and two vertices of $I(G)$ are adjacent iff their corresponding edges are independent; see Fig.~\ref{fig:I(G)}.
\end{definition}

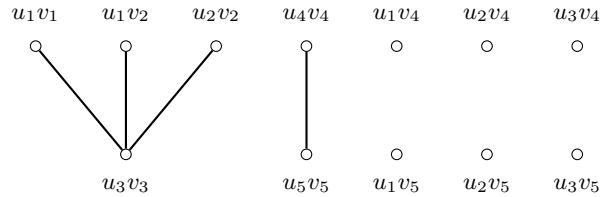
\begin{figure}[t]
\centering\begin{tikzpicture}[scale=1.2]
\useasboundingbox (-0.5, -0.6) rectangle (6.5, 1.2+0.6);

\tikzstyle{every node}=[draw,circle,fill=white,minimum size=4pt,inner sep=0pt]
\node [label=above:$u_1v_1$] (u1v1) at (0, 1.2) {};
\node [label=above:$u_1v_2$] (u1v2) at (1, 1.2) {};
\node [label=above:$u_2v_2$] (u2v2) at (2, 1.2) {};
\node [label=below:$u_3v_3$] (u3v3) at (1, 0) {};
\node [label=above:$u_4v_4$] (u4v4) at (3, 1.2) {};
\node [label=below:$u_5v_5$] (u5v5) at (3, 0) {};
\node [label=above:$u_1v_4$] at (4, 1.2) {};
\node [label=above:$u_2v_4$] at (5, 1.2) {};
\node [label=above:$u_3v_4$] at (6, 1.20) {};
\node [label=below:$u_1v_5$] at (4, 0) {};
\node [label=below:$u_2v_5$] at (5, 0) {};
\node [label=below:$u_3v_5$] at (6, 0) {};
\draw [thick]
	(u1v1) -- (u3v3)
	(u1v2) -- (u3v3)
	(u2v2) -- (u3v3)
	(u4v4) -- (u5v5)
;
\end{tikzpicture}
\caption{
	The independence graph $I(G)$ 
	of the graph $G$ in Fig.~\ref{fig:ext-sunlet4}.
}
\label{fig:I(G)}
\end{figure}

Hell \textit{et al.}~\cite{HHLM20-SIDMA} showed that 
$I(G)$ is transitively orientable if $G$ is a {\twoDORG} 
by using a forbidden substructure called edge-asteroids. 
Proposition~\ref{prop:independent edges} presents a more intuitive proof.
\begin{proposition}\label{prop:independence graph}
The independence graph of every {\twoDORG} is transitively orientable.
\end{proposition}
\begin{proof}
Let $G$ be a {\twoDORG} with bipartition $(U, V)$,
and let $(<_x, <_y)$ be a representation of $G$.
We define an orientation $F$ of $I(G)$ such that
$u_1v_1 \to_F u_2v_2$ if $u_1 <_x u_2$.
Suppose $u_1v_1 \to_F u_2v_2 \to_F u_3v_3$.
Proposition~\ref{prop:independent edges} indicates 
$u_1v_1 \sim_{I(G)} u_3v_3$ and $u_1 <_x u_3$. 
Thus, $F$ is transitive.
\end{proof}

The following theorem indicates that 
the number of {\NR}s of a graph can be obtained 
from the number of transitive orientations of its independence graph. 
\begin{theorem}\label{thm:1-to-1 correspondence 1}
Let $G$ be a {\twoDORG}. 
There is a one-to-one correspondence between 
the {\NR}s of $G$ and transitive orientations of $I(G)$.
\end{theorem}
\begin{proof}
As shown in the proof of Proposition~\ref{prop:independence graph}, 
a transitive orientation of $I(G)$ can be obtained from a representation of $G$. 
It suffices to prove that a {\NR} of $G$ can be obtained from 
every transitive orientation of $I(G)$. 
First, by definition, the relative position of two vertices 
is always the same if they are adjacent, i.e., 
$u <_x v$ and $u <_y v$ if $uv \in E(G)$ 
in any representation $(<_x, <_y)$ of $G$. 
Next, the following claim indicates that the relative position 
of two non-adjacent vertices is the same in every {\NR} 
unless they are incident with independent edges. 
\begin{Claim}\label{claim:relative positions in normalized representation}
The following holds. 
\begin{itemize}
\item Any $u_1, u_2 \in U$ are incident with independent edges iff 
$N(u_1) \not\subset N(u_2)$ and $N(u_1) \not\supset N(u_2)$.
\item Any $v_1, v_2 \in V$ are incident with independent edges iff 
$N(v_1) \not\subset N(v_2)$ and $N(v_1) \not\supset N(v_2)$.
\item Any $u \in U$ and $v \in V$ with $uv \notin E(G)$ are incident with independent edges iff 
there is $v' \in N(u)$ such that $N(v) \not\subset N(v')$.
\end{itemize}
\end{Claim}
\begin{proof}
The proof is straightforward and omitted. 
\end{proof}
Then, we define a set of directed edges $D$ as 
\begin{align*}
D = & \{(u, v) \in U \times V \colon\ uv \in E(G)\} \\
& \cup \{(u_1, u_2) \in U^2 \colon\ N(u_1) \supset N(u_2)\} \\
& \cup \{(v_1, v_2) \in V^2 \colon\ N(v_1) \subset N(v_2)\} \\
& \cup \{(v, u) \in V \times U \colon\ 
\forall v' \in N(u), N(v) \subset N(v')\}. 
\end{align*}
Let $F$ be a transitive orientation of $I(G)$. We define 
\[
D(F) = \{(u_1, u_2), (u_1, v_2), (v_1, u_2), (v_1, v_2) \in V(G)^2 
\colon\ u_1v_1 \to_F u_2v_2\}. 
\]
Claim~\ref{claim:relative positions in normalized representation} 
indicates that $(V(G), D \cup D(F))$ and $(V(G), D \cup D(F)^{-1})$ 
are the orientations of the complete graph on $V(G)$.
Finally, we prove that they are acyclic. 
\begin{Claim}\label{claim:acyclicity}
$(V(G), D \cup D(F))$ and $(V(G), D \cup D(F)^{-1})$ are acyclic. 
\end{Claim}
\begin{proof}
Suppose to the contrary that 
$(V(G), D \cup D(F))$ contains a directed cycle $a \to b \to c \to a$. 
If all three edges are in $D(F)$, 
then $F$ would be a directed cycle, a contradiction. 
Consider the case in which $(a, b) \in D$ and $(b, c), (c, a) \in D(F)$.
We distinguish four cases. 
\begin{itemize}
\item 
Suppose $a \in U$ and $b \in V$, i.e., $ab \in E(G)$. 
As $(b, c) \in D(F)$, there are independent edges $bb'$ and $cc'$ 
such that $bb' \to_F cc'$. 
If $c \in V$ then $ab$ and $cc'$ are independent;
however, $(b, c), (c, a) \in D(F)$, a contradiction. 
Thus, $c \in U$; 
see Fig.~\ref{fig:acyclicity}\subref{fig:acyclicity 1}.
If there is a vertex $v \in N(c) \setminus N(a)$, then 
$ab$ and $cv$ are independent; however, $(b, c), (c, a) \in D(F)$, 
a contradiction. 
Thus, $N(a) \supset N(c)$, contradicting $(c, a) \in D(F)$. 
\item 
Suppose $a, b \in U$, i.e., $N(a) \supset N(b)$. 
As $(b, c) \in D(F)$, there are independent edges $bb'$ and $cc'$ 
such that $bb' \to_F cc'$;
see Fig.~\ref{fig:acyclicity}\subref{fig:acyclicity 2}.
Then, $ab' \in E(G)$ as $N(a) \supset N(b)$. 
Thus, $(V(G), D \cup D(F))$ contains a directed cycle $a \to b' \to c \to a$ 
with $ab' \in E(G)$, a contradiction. 
\item 
Suppose $a, b \in V$, i.e., $N(a) \subset N(b)$. 
As $(c, a) \in D(F)$, there are independent edges $cc'$ and $aa'$ 
such that $cc' \to_F aa'$;
see Fig.~\ref{fig:acyclicity}\subref{fig:acyclicity 3}.
Then, $a'b \in E(G)$ as $N(a) \subset N(b)$.
Thus, $(V(G), D \cup D(F))$ contains a directed cycle $a' \to b \to c \to a'$ 
with $a'b \in E(G)$, a contradiction. 
\item 
Suppose $a \in V$ and $b \in U$. 
As $(b, c) \in D(F)$, there are independent edges $bb'$ and $cc'$ 
such that $bb' \to_F cc'$;
see Fig.~\ref{fig:acyclicity}\subref{fig:acyclicity 4}.
Then, $b' \in V$, $N(a) \subset N(b')$, and $(a, b') \in D$.
Thus, $(V(G), D \cup D(F))$ contains a directed cycle $a \to b' \to c \to a$ 
with $a, b' \in V$, a contradiction. 
\end{itemize}
It is routine to prove that 
if $(a, b), (b, c) \in D$ then $(a, c) \in D$. 
Therefore, $(V(G), D \cup D(F))$ is acyclic. 
Similar arguments would show that $(V(G), D \cup D(F)^{-1})$ is acyclic. 
\end{proof}
As both $(V(G), D \cup D(F))$ and $(V(G), D \cup D(F)^{-1})$ are acyclic, 
they yield two linear orders on $V(G)$. 
It is routine to verify 
that this is a {\NR} of $G$. 
\end{proof}

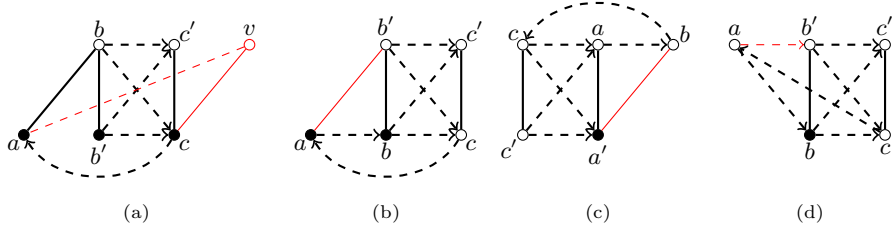
\begin{figure}[t]
\centering\subcaptionbox{\label{fig:acyclicity 1}}{\begin{tikzpicture}
\useasboundingbox (-1.8, -0.7) rectangle (1.8, 1.2+0.7);

\tikzstyle{every node}=[draw,circle,fill=white,minimum size=4pt,inner sep=0pt]
\node [label=below left:$a$,fill=black]  (a)  at (-1.5, 0) {};
\node [label=below:$b'$,fill=black]      (b') at (-0.5, 0) {};
\node [label=below right:$c$,fill=black] (c)  at ( 0.5, 0) {};
\node [label=above:$b$]                  (b)  at (-0.5, 1.2) {};
\node [label=above right:$c'$]           (c') at ( 0.5, 1.2) {};
\draw [thick]
	(a) -- (b)
	(b) -- (b')
	(c) -- (c')
;
\draw [thick,dashed,->] (b) -- (c);
\draw [thick,dashed,->] (b) -- (c');
\draw [thick,dashed,->] (b') -- (c);
\draw [thick,dashed,->] (b') -- (c');
\draw [thick,dashed,->] (c) to [out=240,in=-60] (a);
\node [label=above:$v$,draw=red,] (v) at (1.5, 1.2) {};
\draw [red] 
	(v) -- (c)
;
\draw [dashed,red] 
	(v) -- (a)
;
\end{tikzpicture}}
\centering\subcaptionbox{\label{fig:acyclicity 2}}{\begin{tikzpicture}
\useasboundingbox (-1.3, -0.7) rectangle (1.3, 1.2+0.7);

\tikzstyle{every node}=[draw,circle,fill=white,minimum size=4pt,inner sep=0pt]
\node [label=below left:$a$,fill=black] (a)  at (-1, 0) {};
\node [label=below:$b$,fill=black]      (b)  at ( 0, 0) {};
\node [label=below right:$c$]           (c)  at ( 1, 0) {};
\node [label=above:$b'$]                (b') at ( 0, 1.2) {};
\node [label=above right:$c'$]          (c') at ( 1, 1.2) {};
\draw [thick]
	(b) -- (b')
	(c) -- (c')
;
\draw [thick,dashed,->] (a) -- (b);
\draw [thick,dashed,->] (b) -- (c);
\draw [thick,dashed,->] (b) -- (c');
\draw [thick,dashed,->] (b') -- (c);
\draw [thick,dashed,->] (b') -- (c');
\draw [thick,dashed,->] (c) to [out=240,in=-60] (a);
\draw [red] 
	(a) -- (b')
;
\end{tikzpicture}}
\centering\subcaptionbox{\label{fig:acyclicity 3}}{\begin{tikzpicture}
\useasboundingbox (-1.3, -0.7) rectangle (1.3, 1.2+0.7);

\tikzstyle{every node}=[draw,circle,fill=white,minimum size=4pt,inner sep=0pt]
\node [label=below left:$c'$]       (c') at (-1, 0) {};
\node [label=below:$a'$,fill=black] (a') at ( 0, 0) {};
\node [label=above left:$c$]        (c)  at (-1, 1.2) {};
\node [label=above:$a$]             (a)  at ( 0, 1.2) {};
\node [label=above right:$b$]       (b)  at ( 1, 1.2) {};
\draw [thick]
	(a) -- (a')
	(c) -- (c')
;
\draw [thick,dashed,->] (a) -- (b);
\draw [thick,dashed,->] (c) -- (a);
\draw [thick,dashed,->] (c) -- (a');
\draw [thick,dashed,->] (c') -- (a);
\draw [thick,dashed,->] (c') -- (a');
\draw [thick,dashed,->] (b) to [out=120,in=60] (c);
\draw [red] 
	(a') -- (b)
;
\end{tikzpicture}}
\centering\subcaptionbox{\label{fig:acyclicity 4}}{\begin{tikzpicture}
\useasboundingbox (-1.3, -0.7) rectangle (1.3, 1.2+0.7);

\tikzstyle{every node}=[draw,circle,fill=white,minimum size=4pt,inner sep=0pt]
\node [label=below:$b$,fill=black] (b)  at ( 0, 0) {};
\node [label=below:$c$]            (c)  at ( 1, 0) {};
\node [label=above:$a$]            (a)  at (-1, 1.2) {};
\node [label=above:$b'$]           (b') at ( 0, 1.2) {};
\node [label=above:$c'$]           (c') at ( 1, 1.2) {};
\draw [thick]
	(b) -- (b')
	(c) -- (c')
;
\draw [thick,dashed,->] (a) -- (b);
\draw [thick,dashed,->] (c) -- (a);
\draw [thick,dashed,->] (b) -- (c);
\draw [thick,dashed,->] (b) -- (c');
\draw [thick,dashed,->] (b') -- (c);
\draw [thick,dashed,->] (b') -- (c');
\draw [red,dashed,->] (a) -- (b');
\end{tikzpicture}}
\caption{Illustrations for the proof of Claim~\ref{claim:acyclicity}.} 
\label{fig:acyclicity}
\end{figure}

The following notion is a key concept in studying comparability graphs. 
\begin{definition}[\cite{Golumbic04}]
Let $G$ be a graph and 
$Z_G = \{(a, b) \colon\ ab \in E(G)\}$ be the set of ordered pairs of the vertices in $G$. 
The \emph{$\Gamma$-relation} is a binary relation on $Z_G$ such that 
$a_1b_1 \Gamma a_2b_2$ iff 
either $a_1 = a_2$ and $b_1b_2 \notin E(G)$ 
or $b_1 = b_2$ and $a_1a_2 \notin E(G)$. 
The $\Gamma$-relation is symmetric, and 
its reflexive and transitive closure $\Gamma^*$ is 
an equivalence relation on $Z_G$. 
Each equivalence class of $Z_G$ under $\Gamma^*$ 
is called an \emph{implication class} of $G$. 
For an implication class $A$, 
the \emph{reversal} of $A$ is the set 
$A^{-1} = \{(a, b) \colon\ (b, a) \in A\}$ 
of ordered pairs of $Z_G$. 
The reversal of any implication class is 
also an implication class. 
If the graph is transitively orientable, 
then $A \neq A^{-1}$ for every implication class $A$. 
\end{definition}

The following theorem provides the relations 
between transitive orientations of $I(G)$ 
and non-trivial components of $G^+$. 
\begin{theorem}\label{thm:1-to-1 correspondence 2}
Let $G$ be a {\twoDORG}. 
There is a one-to-one correspondence between 
implication classes of $I(G)$ and non-trivial components of $G^+$.
\end{theorem}
\begin{proof}
For each edge $u_1v_1 \sim_{I(G)} u_2v_2$ in $I(G)$, 
the auxiliary graph $G^+$ contains 
a pair of edges $(u_1, v_2) \sim_{G^+} (v_1, u_2)$ 
and $(u_2, v_1) \sim_{G^+} (v_2, u_1)$ and vice versa; 
see Figs.~\ref{fig:G^+} and~\ref{fig:I(G)}. 
We associate an ordered pair $(u_1v_1, u_2v_2)$ in 
an implication class of $I(G)$ 
to an edge $(u_1, v_2) \sim_{G^+} (v_1, u_2)$ 
and, symmetrically, $(u_2v_2, u_1v_1)$ to $(u_2, v_1) \sim_{G^+} (v_2, u_1)$. 
It suffices to show that for any two edges 
$u_1v_1 \sim_{I(G)} u_2v_2$ and $u_3v_3 \sim_{I(G)} u_4v_4$, 
we have $(u_1v_1, u_2v_2) \Gamma^* (u_3v_3, u_4v_4) \iff (u_1, v_2) \sim_{G^+}^* (u_3, v_4)$, 
where $(u_1, v_2) \sim_{G^+}^* (u_3, v_4)$ denotes that 
there is a path from $(u_1, v_2)$ to $(u_3, v_4)$ in $G^+$.
\par
Suppose that there are $\Gamma$-related edges 
in $I(G)$ such that $(u_1v_1, u_2v_2) \Gamma$ $(u_3v_3, u_2v_2)$;
see Fig.~\ref{fig:1-to-1 2}\subref{fig:1-to-1 2-1}. 
As $u_1v_1$ and $u_3v_3$ are not independent, 
$u_1v_3 \in E(G)$ or $u_3v_1 \in E(G)$. 
If $u_1v_3 \in E(G)$ then 
$(u_1, v_2) \sim_{G^+} (v_3, u_2) \sim_{G^+} (u_3, v_2)$;
if $u_3v_1 \in E(G)$ then 
$(u_1, v_2) \sim_{G^+} (v_1, u_2) \sim_{G^+} (u_3, v_2)$.
In both cases, $(u_1, v_2) \sim_{G^+}^* (u_3, v_2)$.
Thus, for any two edges 
$u_1v_1 \sim_{I(G)} u_2v_2$ and $u_3v_3 \sim_{I(G)} u_4v_4$, 
we have $(u_1v_1, u_2v_2)$ $\Gamma^* (u_3v_3, u_4v_4) \implies 
(u_1, v_2) \sim_{G^+}^* (u_3, v_4)$.
\par
Conversely, suppose that there is a path 
$(u_1, v_1) \sim_{G^+} (v_2, u_2) \sim_{G^+} (u_3, v_3)$ in $G^+$; 
see Fig.~\ref{fig:1-to-1 2}\subref{fig:1-to-1 2-2}. 
We have $u_1v_3 \notin E(G)$ or $u_3v_1 \notin E(G)$; otherwise, 
$G$ contains an induced cycle of length 6, contradicting that 
any {\twoDORG} is chordal bipartite~\cite{STU10-DAM}.
If $u_1v_3 \notin E(G)$ then 
$(u_1v_2, u_2v_1) \Gamma$ $(u_1v_2, u_2v_3) \Gamma (u_3v_2, u_2v_3)$; 
if $u_3v_1 \notin E(G)$ then 
$(u_1v_2, u_2v_1) \Gamma (u_3v_2, u_2v_1) \Gamma$ $(u_3v_2,$ $u_2v_3)$. 
In both cases, $(u_1v_2, u_2v_1) \Gamma^* (u_3v_2, u_2v_3)$. 
Thus, for any two edges 
$u_1v_1$ $\sim_{I(G)} u_2v_2$ and $u_3v_3 \sim_{I(G)} u_4v_4$, 
we have $(u_1, v_2) \sim_{G^+}^* (u_3, v_4) \implies 
(u_1v_1, u_2v_2)$ $\Gamma^* (u_3v_3, u_4v_4)$. 
\end{proof}

\begin{figure}[t]
\centering\subcaptionbox{\label{fig:1-to-1 2-1}}{\begin{tikzpicture}
\useasboundingbox (-1.5, -2.0) rectangle (4.5, 1.4);

\tikzstyle{every node}=[draw,circle,fill=white,minimum size=4pt,inner sep=0pt]
\node [label=right:$u_3$,fill=black] (u3) at (0:1) {};
\node [label=above:$v_3$]            (v3) at (60:1) {};
\node [label=above:$u_1$,fill=black] (u1) at (120:1) {};
\node [label=left:$v_1$]             (v1) at (180:1) {};
\node [label=below:$u_2$,fill=black] (u2) at (240:1) {};
\node [label=below:$v_2$]            (v2) at (300:1) {};
\draw [thick]
	(u1) -- (v1)
	(u2) -- (v2)
	(u3) -- (v3)
;
\draw [thick,dashed] 
	(u1) -- (v2)
	(u2) -- (v1)
	(u2) -- (v3)
	(u3) -- (v2)
;
\draw [red]
	(u1) -- (v3)
;

\node [label=above left :$u_1v_1$] (u1v1) at ($0.5*(u1)+0.5*(v1)+(3,0.25)$) {};
\node [label=below:$u_2v_2$      ] (u2v2) at ($0.5*(u2)+0.5*(v2)+(3,0.25)$) {};
\node [label=above right:$u_3v_3$] (u3v3) at ($0.5*(u3)+0.5*(v3)+(3,0.25)$) {};
\draw [thick]
	(u1v1) -- (u2v2)
	(u3v3) -- (u2v2)
;
\draw [thick,dashed]
	(u1v1) -- (u3v3)
;
\draw (0, -1.8) node [draw=none] {$G$};
\draw (3, -1.8) node [draw=none] {$I(G)$};
\end{tikzpicture}}
\centering\subcaptionbox{\label{fig:1-to-1 2-2}}{\begin{tikzpicture}
\useasboundingbox (-1.5, -2.0) rectangle (4.3, 1.4);

\tikzstyle{every node}=[draw,circle,fill=white,minimum size=4pt,inner sep=0pt]
\node [label=below:$u_1$,fill=black] (u1) at (-1.2,-0.6) {};
\node [label=above:$v_1$]            (v1) at (-1.2, 0.6) {};
\node [label=above:$u_2$,fill=black] (u2) at (   0, 0.6) {};
\node [label=below:$v_2$]            (v2) at (   0,-0.6) {};
\node [label=below:$u_3$,fill=black] (u3) at ( 1.2,-0.6) {};
\node [label=above:$v_3$]            (v3) at ( 1.2, 0.6) {};
\draw [thick,dashed] 
	(u1) -- (v1)
	(u2) -- (v2)
	(u3) -- (v3)
;
\draw [thick]
	(u1) -- (v2) -- (u3)
	(v1) -- (u2) -- (v3)
;
\draw [dashed,red]
	(u1) -- (v3)
;

\node [label=above:$u_2v_1$] (u2v1) at ($(-0.6, 0.6)+(2.2,0)+(1,0)$) {};
\node [label=above:$u_2v_3$] (u2v3) at ($( 0.6, 0.6)+(2.2,0)+(1,0)$) {};
\node [label=below:$u_1v_2$] (u1v2) at ($(-0.6,-0.6)+(2.2,0)+(1,0)$) {};
\node [label=below:$u_3v_2$] (u3v2) at ($( 0.6,-0.6)+(2.2,0)+(1,0)$) {};
\draw [thick]
	(u2v1) -- (u1v2)
	(u2v3) -- (u3v2)
;
\draw [red]
	(u1v2) -- (u2v3)
;
\draw [thick,dashed]
	(u2v1) -- (u2v3)
	(u1v2) -- (u3v2)
;
\draw (0, -1.8) node [draw=none] {$G$};
\draw (3, -1.8) node [draw=none] {$I(G)$};
\end{tikzpicture}}
\caption{Illustrations for the proof of Theorem~\ref{thm:1-to-1 correspondence 2}.} 
\label{fig:1-to-1 2}
\end{figure}
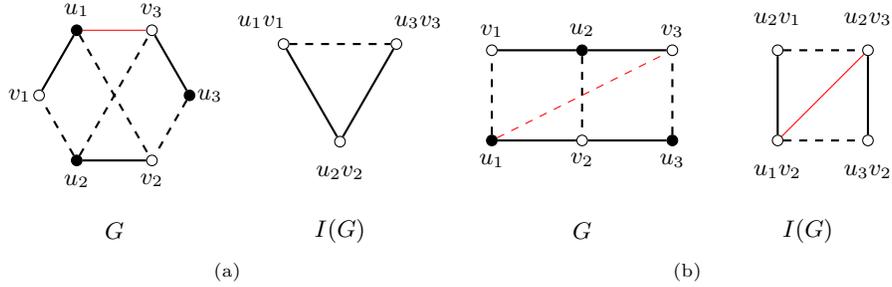

Now, we are ready to prove the equivalence (i) $\iff$ (iii) of Theorem~\ref{thm:main}.
A comparability graph has two transitive orientations, one being the reversal of the other, iff it admits two implication classes~\cite[Corollary 5.5]{Golumbic04}.
Therefore, Theorems~\ref{thm:1-to-1 correspondence 1} and~\ref{thm:1-to-1 correspondence 2} indicate that a {\twoDORG} $G$ is uniquely representable iff its auxiliary graph $G^+$ has two non-trivial components.

\subsubsection{Chain Graphs and Disconnected Graphs.}
Theorems~\ref{thm:chain graph} and~\ref{thm:disconnected} are corollaries of Theorem~\ref{thm:1-to-1 correspondence 1}, and we present the proofs here. 

\begin{theorem*}[Theorem~\ref{thm:chain graph}]
A {\twoDORG} has exactly one {\NR} iff it is a chain graph.
\end{theorem*}
\begin{proof}
A comparability graph has at least two transitive orientations if it has an edge, 
one being the reversal of the other. 
Thus, a {\twoDORG} $G$ has at least two {\NR}s 
if the independence graph $I(G)$ has an edge. 
As $I(G)$ has no edge iff $G$ is a chain graph, the theorem holds. 
\end{proof}

\begin{theorem*}[Theorem~\ref{thm:disconnected}]
A disconnected {\twoDORG} is uniquely representable iff 
it has two non-trivial components 
each of which is a chain graph. 
\end{theorem*}
\begin{proof}
Let $G$ be a disconnected {\twoDORG}. 
It suffices to show that $G$ 
has two non-trivial components each of which is a chain graph 
iff the independence graph $I(G)$ has exactly two transitive orientations. 
If $G$ contains two non-trivial components (i.e., each component has an edge), 
then $I(G)$ has an edge. 
Thus, $I(G)$ has at least two transitive orientations. 
Conversely, if $G$ contains more than two non-trivial components 
or one of two non-trivial components is not a chain graph 
(i.e., the component has independent edges), 
then $I(G)$ has at least four implication classes. 
Thus, $I(G)$ has more than two transitive orientations. 
\end{proof}

\subsection{Proof of (ii) $\iff$ (iii)}\label{sec:(ii) <==> (iii)}
In this section, we prove the equivalence (ii) $\iff$ (iii) of Theorem~\ref{thm:main}. 
The following is a fundamental fact about buried subgraphs.
\begin{proposition}\label{prop:K set}
Let $G$ be a {\twoDORG} with bipartition $(U, V)$ and 
$B$ be its buried subgraph. Let 
$K_U(B) = \{u \in U \setminus B \colon\ u$ is adjacent to all of $B \cap V\}$ and 
$K_V(B) = \{v \in V \setminus B \colon\ v$ is adjacent to all of $B \cap U\}$.
The set $K_U(B) \cup K_V(B)$ induces a biclique. 
\end{proposition}
\begin{proof}
Suppose to the contrary that there are non-adjacent vertices 
$u \in K_U(B)$ and $v \in K_V(B)$. 
By definition, $G[B]$ contains independent edges. 
As $u$ and $v$ are adjacent to the endvertices of both edges, 
$G$ contains an induced cycle of length 6, contradicting that 
every {\twoDORG} is chordal bipartite~\cite{STU10-DAM}.
\end{proof}

The following two lemmas yield the equivalence (ii) $\iff$ (iii) of Theorem~\ref{thm:main}.
\begin{lemma}\label{lem:(ii) <== (iii)}
If a {\twoDORG} $G$ contains a buried subgraph $B$, then 
the auxiliary graph $G^+$ has more than two non-trivial components. 
\end{lemma}
\begin{proof}
Let $(U, V)$ be the bipartition of $G$. 
Due to Conditions (a) and (c) of Definition~\ref{def:buried subgraph}, 
it suffices to show that in $G^+$, 
no vertex inside $G[B]$ is adjacent to a vertex outside $G[B]$. 
Suppose to the contrary that 
there is a vertex $(u_1, v_1) \in V(G^+)$ inside $G[B]$ adjacent to $(v_2, u_2) \in V(G^+)$ outside $G[B]$, that is, 
$u_1v_1, u_2v_2 \notin E(G)$, $u_1v_2, u_2v_1 \in E(G)$, $u_1, v_1 \in B$, and $u_2$ or $v_2$ is not in $B$. 
If $u_2, v_2 \notin B$ then 
Proposition~\ref{prop:K set} indicates $u_2v_2 \in E(G)$, a contradiction. 
If $u_2 \notin B$ and $v_2 \in B$, then $u_2$ is adjacent to 
$v_1 \in B$ but not to $v_2 \in B$, 
contradicting Condition (b) of Definition~\ref{def:buried subgraph}. 
Similarly, a contradiction arises when $u_2 \in B$ and $v_2 \notin B$. 
\end{proof}

\begin{lemma}\label{lem:(ii) ==> (iii)}
A {\twoDORG} $G$ contains a buried subgraph 
if the auxiliary graph $G^+$ has more than two non-trivial components. 
\end{lemma}
\begin{proof}
Let $(U, V)$ be the bipartition of $G$ and 
$(<_U, <_V)$ be a normalized weak ordering of $G$. 
Let $u_{\alpha}v_{\alpha}$ and $u_{\beta}v_{\beta}$ be independent edges 
such that $u_{\alpha}$ is minimum {\wrt} $<_U$, and subject to this condition, 
$v_{\alpha}$ is minimum {\wrt} $<_V$, that is, 
the ordered pair $(u_{\alpha}, v_{\alpha})$ is lexicographically smallest. 
Let $C_{\alpha}$ be the component of $G^+$ 
containing $(u_{\alpha}, v_{\alpha})$ 
and $C_{\alpha}^{-1}$ be its reversed component. 
As $G^+$ contains more than two components, 
$G^+$ contains another component $C$. 
Let $C^{-1}$ be the reversed component of $C$, 
and let $\hat{C} = C \cup C^{-1}$.
\par
We define $u_{\ell}$ as the vertex of $U$ minimum {\wrt} $<_U$ 
among the vertices appearing in $\hat{C}$. 
In other words, $u_{\ell}$ is the vertex such that 
there are independent edges $u_{\ell}v_1$ and $u_2v_2$, 
one is incident with $u_{\ell}$, such that 
$u_{\ell} <_U u_2$, $v_1 <_V v_2$, and 
$(u_{\ell}, v_2) \in \hat{C}$, and subject to these conditions, 
$u_{\ell}$ is minimum {\wrt} $<_U$. 
We also define $u_r$ as the vertex of $U$ maximum {\wrt} $<_U$ 
among the vertices appearing in $\hat{C}$. 
Moreover, two vertices $v_{\ell}$ and $v_r$ of $V$ can be defined similarly. 
\par
Now, we define a vertex set $B$ as 
\[B = \{u \in U \colon\ u_{\ell} \leq_U u \leq_U u_r\} \cup 
\{v \in V \colon\ v_{\ell} \leq_V v \leq_V v_r\}\]
and prove that $B$ is a buried subgraph. 
First, we claim the following. 
\begin{Claim}\label{main technical claim}
The following holds for the set $B$. 
\begin{itemize}
\item No vertex $u <_U u_{\ell}$ is adjacent to any of $B \cap V$. 
\item No vertex $v <_V v_{\ell}$ is adjacent to any of $B \cap U$. 
\item Each vertex $u >_U u_r$ is adjacent to either all of $B \cap V$ or none of them. 
\item Each vertex $v >_V v_r$ is adjacent to either all of $B \cap U$ or none of them. 
\end{itemize}
\end{Claim}

Claim~\ref{main technical claim} can be proven by the following series of claims. 
\begin{Claim}\label{technical claim 1}
Let $u_1v_1$ and $u_2v_2$ be independent edges such that 
$u_1 <_U u_2$, $v_1 <_V v_2$, and 
$(u_1, v_2) \in \hat{C}$, and 
let $u$ be a vertex with $u <_U u_1$. 
If $(u, v_2) \notin \hat{C}$ then 
$u$ is not adjacent to any vertex $v \in V$ 
with $v_1 \leq_V v \leq_V v_2$. 
\end{Claim}
\begin{proof}
Suppose to the contrary that there is a vertex $v \in N(u)$ with 
$v_1 \leq_V v \leq_V v_2$; 
see Fig.~\ref{fig:(ii) ==> (iii)}\subref{fig:Claim 4}. 
We have $uv_2 \notin E(G)$; otherwise, 
$u_1v_1 \in E(G)$ indicates $u_1v_2 \in E(G)$ 
because of the property of weak ordering. 
Then, $uv_1 \notin E(G)$; otherwise, 
$(v_1, u_2) \sim (u, v_2)$ in $G^+$ 
and $(u, v_2) \in \hat{C}$, a contradiction. 
\par
Now, we prove $N(v) \supset N(v_2)$. 
Suppose to the contrary that there is a vertex $u' \in N(v_2) \setminus N(v)$. 
We have $u' >_U u_1$; otherwise, 
$u'v_2, u_1v_1 \in E(G)$ indicates $u_1v_2 \in E(G)$ 
because of the property of weak ordering. 
As $uv, u_1v_1 \in E(G)$, we have $u_1v \in E(G)$. 
Now, $(u_1, v_2) \sim (v, u') \sim (u, v_2)$ in $G^+$ 
and $(u, v_2) \in \hat{C}$, a contradiction. 
Thus, $N(v) \supseteq N(v_2)$. 
Clearly, $N(v) \neq N(v_2)$ as $u, u_1 \in N(v) \setminus N(v_2)$. 
However, $N(v) \supset N(v_2)$ and $v <_V v_2$ 
contradict the property of normalized weak ordering. 
\end{proof}

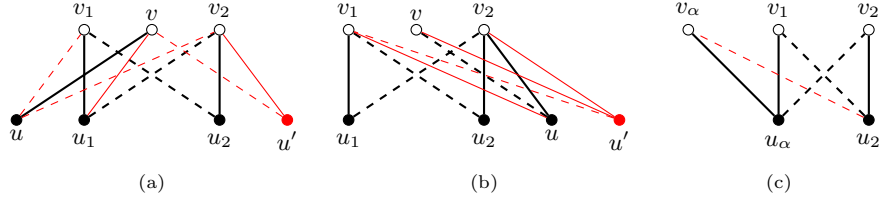
\begin{figure}[t]
\centering\subcaptionbox{\label{fig:Claim 4}}{\begin{tikzpicture}
\useasboundingbox (-2.1, -0.5) rectangle (2.1, 1.7);

\tikzstyle{every node}=[draw,circle,fill=white,minimum size=4pt,inner sep=0pt]
\node [label=below:$u$,fill=black]   (u)  at (-1.8,0) {};
\node [label=below:$u_1$,fill=black] (u1) at (-0.9,0) {};
\node [label=below:$u_2$,fill=black] (u2) at ( 0.9,0) {};
\node [label=above:$v_1$] (v1) at (-0.9,1.2) {};
\node [label=above:$v$]   (v)  at (   0,1.2) {};
\node [label=above:$v_2$] (v2) at ( 0.9,1.2) {};
\draw [thick]
	(u1) -- (v1)
	(u2) -- (v2)
	(u) -- (v)
;
\draw [thick,dashed] 
	(u1) -- (v2)
	(u2) -- (v1)
;
\draw [dashed,red] 
	(u) -- (v2)
	(u) -- (v1)
;
\node [label=below:$u'$,draw=red,fill=red] (u') at (1.8,0) {};
\draw [dashed,red] (u') -- (v);
\draw [red] 
	(u') -- (v2)
	(u1) -- (v)
;
\end{tikzpicture}}
\centering\subcaptionbox{\label{fig:Claim 5}}{\begin{tikzpicture}
\useasboundingbox (-2.1, -0.5) rectangle (2.1, 1.7);

\tikzstyle{every node}=[draw,circle,fill=white,minimum size=4pt,inner sep=0pt]
\node [label=below:$u_1$,fill=black] (u1) at (-1.8,0) {};
\node [label=below:$u_2$,fill=black] (u2) at (   0,0) {};
\node [label=below:$u$,fill=black]   (u)  at ( 0.9,0) {};
\node [label=above:$v_1$] (v1) at (-1.8,1.2) {};
\node [label=above:$v$]   (v)  at (-0.9,1.2) {};
\node [label=above:$v_2$] (v2) at (   0,1.2) {};
\draw [thick]
	(u1) -- (v1)
	(u2) -- (v2)
	(u) -- (v2)
;
\draw [thick,dashed] 
	(u1) -- (v2)
	(u2) -- (v1)
	(u) -- (v)
;
\draw [red] 
	(u) -- (v1)
;
\node [label=below:$u'$,draw=red,fill=red] (u') at (1.8,0) {};
\draw [dashed,red] (u') -- (v1);
\draw [red] 
	(u') -- (v)
	(u') -- (v2)
;
\end{tikzpicture}}
\centering\subcaptionbox{\label{fig:Lemma 38}}{\begin{tikzpicture}
\useasboundingbox (-1.6, -0.5) rectangle (1.6, 1.7);

\tikzstyle{every node}=[draw,circle,fill=white,minimum size=4pt,inner sep=0pt]
\node [label=below:$u_{\alpha}$,fill=black] (u1) at (   0,0) {};
\node [label=below:$u_2$,fill=black]        (u2) at ( 1.2,0) {};
\node [label=above:$v_1$]                   (v1) at (   0,1.2) {};
\node [label=above:$v_2$]                   (v2) at ( 1.2,1.2) {};
\node [label=above:$v_{\alpha}$]            (v)  at (-1.2,1.2) {};
\draw [thick]
	(u1) -- (v1)
	(u2) -- (v2)
	(u1) -- (v)
;
\draw [thick,dashed] 
	(u1) -- (v2)
	(u2) -- (v1)
;
\draw [red,dashed] 
	(u2) -- (v)
;
\end{tikzpicture}}
\caption{Illustrations for the proof of Lemma~\ref{lem:(ii) ==> (iii)}.} 
\label{fig:(ii) ==> (iii)}
\end{figure}

\begin{Claim}\label{technical claim 2}
Let $u_1v_1$ and $u_2v_2$ be independent edges such that 
$u_1 <_U u_2$, $v_1 <_V v_2$, and 
$(u_1, v_2) \in \hat{C}$, and 
let $u$ be a vertex with $u >_U u_2$. 
If there is no vertex $v \in V$ such that $(u, v) \in \hat{C}$, then 
$u$ is adjacent to all vertices $v$ with $v_1 \leq_V v \leq_V v_2$ or 
none of them. 
\end{Claim}
\begin{proof}
Suppose to the contrary that 
there are two vertices $v, v' \in V$ such that $v_1 \leq_V \{v, v'\} \leq_V v_2$, 
$uv \notin E(G)$, and $uv' \in E(G)$. 
Because of the property of weak ordering, 
we have $uv_2 \in E(G)$ from $uv', u_2v_2 \in E(G)$; 
see Fig.~\ref{fig:(ii) ==> (iii)}\subref{fig:Claim 5}. 
Then, $uv_1 \in E(G)$; otherwise, 
$(u_1, v_2) \sim (v_1, u)$ in $G^+$ and $(u, v_1) \in \hat{C}$, a contradiction. 
\par
Now, we prove $N(v_1) \supset N(v)$. 
Suppose to the contrary that there is a vertex $u' \in N(v) \setminus N(v_1)$. 
We have $u' >_U u$; otherwise, $u'v, uv_1 \in E(G)$ indicates $uv \in E(G)$ 
because of the property of weak ordering. 
Then, $u'v, uv_2 \in E(G)$ indicates $u'v_2 \in E(G)$. 
Now, $(u_1, v_2) \sim (v_1, u') \sim (u, v)$ in $G^+$ 
and $(u, v) \in \hat{C}$, a contradiction. 
Thus, $N(v_1) \supseteq N(v)$. 
Clearly, $N(v_1) \neq N(v)$ as $u \in N(v_1) \setminus N(v)$. 
However, $N(v_1) \supset N(v)$ and $v_1 <_V v$ 
contradict the property of normalized weak ordering. 
\end{proof}

\begin{Claim}\label{technical claim 3}
For any vertex $v \in B \cap V$ other than $v_r$, 
there are two independent edges 
$u_1v_1$ and $u_2v_2$ such that $v_1 \leq_V v <_V v_2$ 
and $(u_1, v_2) \in \hat{C}$. 
\end{Claim}
\begin{proof}
By the definition of $v_{\ell}$, 
there are independent edges $u_av_{\ell}$ and 
$u_a'v_a'$, one is incident with $v_{\ell}$, 
such that $u_a <_U u_a'$, $v_{\ell} <_V v_a'$, 
and $(u_a, v_a'), (v_{\ell}, u_a') \in \hat{C}$. 
Similarly, by the definition of $v_r$, 
there are independent edges $u_bv_b$ and 
$u_b'v_r$, one is incident with $v_r$, 
such that $u_b <_U u_b'$, $v_b <_V v_r$, 
and $(u_b, v_r), (v_b, u_b') \in \hat{C}$. 
If $v_a' >_V v$ or $v_b \leq_V v$, the claim holds. 
\par
Now, consider the case in which $v_a' \leq_V v <_V v_b$. 
As both $(u_a, v_a')$ and $(v_b, u_b')$ are in $\hat{C}$, 
the graph $G^+$ contains a path
$(u_a, v_a') = (u_1, v_1) \sim (v_2, u_2) \sim (u_3, v_3) \sim \cdots \sim (v_k, u_k) = (v_b, u_b')$
if $(u_a, v_a'), (v_b, u_b') \in C$, and 
$(u_a, v_a') = (u_1, v_1) \sim (v_2, u_2) \sim (u_3, v_3) \sim \cdots \sim (u_k, v_k) = (u_b', v_b)$ 
if $(u_a, v_a') \in C$ and $(v_b, u_b') \in C^{-1}$. 
In both cases, as $v_1 = v_a' \leq_V v <_V v_b = v_k$, 
there is an index $i$ such that 
$(u_i, v_i) \sim (v_{i+1}, u_{i+1})$ or
$(v_i, u_i) \sim (u_{i+1}, v_{i+1})$ 
with $v_i \leq_V v <_V v_{i+1}$. 
In other words, there are independent edges 
$u_iv_{i+1}$ and $u_{i+1}v_i$ such that 
$v_i \leq_V v <_V v_{i+1}$ 
and $(u_i, v_i), (v_{i+1}, u_{i+1}) \in \hat{C}$. 
\end{proof}

We are now ready to prove Claim~\ref{main technical claim}.
\begin{proof}[Proof of Claim~\ref{main technical claim}]
To prove the first statement, 
suppose to the contrary that 
there is a vertex $u <_U u_{\ell}$ adjacent to $v$ with $v_{\ell} \leq_V v \leq_V v_r$. 
By Claim~\ref{technical claim 3}, 
there are independent edges 
$u_1v_1$ and $u_2v_2$ such that $v_1 \leq_V v \leq_V v_2$ 
and $(u_1, v_2) \in \hat{C}$. 
As $u <_U u_{\ell}$, we have $(u, v_2) \notin \hat{C}$, 
contradicting Claim~\ref{technical claim 1}. 
Thus, the first statement holds. 
The second statement can be proven in a symmetric manner. 
\par
Now, we prove the third statement. 
Suppose that some vertex in $B \cap V$ is adjacent to $u >_U u_r$. 
As there are two edges $uv'$ and $u'v_r$ 
such that $v' \leq_V v_r$ and $u' \leq_U u_r <_U u$, 
we have $uv_r \in E(G)$ 
from the property of weak ordering. 
We prove by induction that all of $B \cap V$ are adjacent to $u$. 
Let $v <_V v_r$ be a vertex in $B \cap V$. 
We assume by the induction hypothesis that every $v'$ with $v <_V v' \leq_V v_r$ 
is adjacent to $u$. 
By Claim~\ref{technical claim 3}, 
there are independent edges 
$u_1v_1$ and $u_2v_2$ such that $v_1 \leq_V v <_V v_2$ 
and $(u_1, v_2) \in \hat{C}$. 
Then, $v_2 \in N(u)$ according to the induction hypothesis 
and $v \in N(u)$ by Claim~\ref{technical claim 2}. 
Thus, the third statement holds. 
The fourth statement can be proven in a symmetric manner. 
\end{proof}

Finally, we prove that the set $B$ is a buried subgraph. 
By definition, $B$ satisfies Condition (a) of Definition~\ref{def:buried subgraph}. 
Claim~\ref{main technical claim} indicates that 
$B$ satisfies Condition (b). 
For Condition (c), it suffices to show that 
$u_{\alpha}v_{\alpha}$ is outside $G[B]$. 
From the definition of $u_{\alpha}$, we have $u_{\alpha} \leq_U u_{\ell}$. 
Suppose to the contrary that $u_{\alpha} = u_{\ell}$. 
Then, there are independent edges $u_{\alpha}v_1$ and $u_2v_2$ 
such that $u_{\alpha} <_U u_2$, $v_1 <_V v_2$, and 
$(u_{\alpha}, v_2) \in \hat{C}$. 
From the definition of $u_{\alpha}v_{\alpha}$, we have $v_{\alpha} <_V v_1$; 
see Fig.~\ref{fig:(ii) ==> (iii)}\subref{fig:Lemma 38}. 
We have $u_2v_{\alpha} \notin E(G)$; otherwise, 
$u_{\alpha}v_1 \in E(G)$ indicates $u_2v_1 \in E(G)$ 
because of the property of weak ordering. 
Now, $(u_{\alpha}, v_2) \sim (v_{\alpha}, u_2)$ in $G^+$ and $(v_{\alpha}, u_2) \in \hat{C}$, 
contradicting the definition of $u_{\alpha}v_{\alpha}$. 
Therefore, $u_{\alpha} \neq u_{\ell}$, i.e., $u_{\alpha} <_U u_{\ell}$. 
Then, Claim~\ref{main technical claim} indicates 
$v_{\alpha} <_V v_{\ell}$ or $v_{\alpha} >_V v_r$. 
Thus, the set $B$ satisfies Condition (c) of Definition~\ref{def:buried subgraph}. 
We can prove that $B$ satisfies Condition (d) 
by using the property of normalized weak ordering. 
Suppose to the contrary that 
there is a vertex $v$ of $B \cap V$ not adjacent to any of $B \cap U$. 
Then, Claim~\ref{main technical claim} indicates $N(v_{\ell}) \supset N(v)$; 
however, $v_{\ell} <_V v$ contradicts the property of normalized weak ordering. 
Similarly, if there is a vertex $v$ of $B \cap V$ adjacent to all of $B \cap U$, then Claim~\ref{main technical claim} indicates $N(v) \supset N(v_r)$, 
contradicting $v <_V v_r$. 
Symmetric arguments would show that there is no vertex of $B \cap U$ 
that is adjacent to all or no vertices of $B \cap V$. 
Thus, the set $B$ satisfies Condition (d) of Definition~\ref{def:buried subgraph}. 
\end{proof}

\section{Concluding Remarks}\label{sec:conclusion}
We studied the relations between the graph structure and intersection representations of {\twoDORG}s. 
In this paper, we provided a characterization of uniquely representable {\twoDORG}s. 
Following the characterization of uniquely representable interval graphs~\cite{Hanlon82-TAMS,Fishburn85-DAM,FM22-DM}, we introduced the notion of a buried subgraph of a {\twoDORG} and showed that its absence is a necessary and sufficient condition for a graph to be uniquely representable.

We now briefly discuss the substitution of buried subgraphs. 
In an interval graph $G$, 
its buried subgraph $B$ can be substituted by a single vertex $v$ in $B$, 
that is, $G - (B \setminus \{v\})$ is an interval graph.
By substituting all the maximal buried subgraphs, 
we obtain a uniquely representable interval graph~\cite{Hanlon82-TAMS}. 
By definition, 
the representative $v$ is simplicial in the graph obtained after substitution 
(recall that a vertex is simplicial if its neighborhood induces a clique~\cite{Golumbic04}). 
Conversely, we can replace a simplicial vertex with another interval graph. 
These facts play important roles in, for example, counting interval graphs~\cite{Hanlon82-TAMS}. 

In a {\twoDORG}s, it is natural to substitute a buried subgraph with an edge;
see Fig.~\ref{fig:substitution}.
It is straightforward to see that 
a buried subgraph $B$ of a {\twoDORG} $G$ can be substituted by a single edge $uv$ in $G[B]$. 
In the graph obtained after substitution, 
the edge $uv$ is simplicial (i.e., the neighborhood of its endpoints induces a biclique~\cite{Golumbic04}).
Conversely, we can replace a simplicial edge with another {\twoDORG}, 
which will be a buried subgraph in the whole graph obtained after replacement. 

\begin{figure}[t]
\centering\begin{tikzpicture}[scale=0.4]
\draw [dashed,draw=black,fill=black!10!white] (0.5, -4) -- (6.5, -4) -- (6.5, 0.5) -- (11, 0.5) -- (11, 6.5) -- (0.5, 6.5) -- cycle;

\foreach \x in {-3, -2, -1, 0, 1, 2, 3, 4, 5, 6, 7, 8, 9, 10, 11}
    \draw[very thin,dashed] (\x, -4) -- (\x, 10);
\foreach \y in {-4, -3, -2, -1, 0, 1, 2, 3, 4, 5, 6, 7, 8, 9, 10}
    \draw[very thin,dashed] (-3, \y) -- (11, \y);

\tikzstyle{every node}=[draw,circle,fill=white,minimum size=4pt,inner sep=0pt]
\node [fill=black] (u1) at (1, 3) {};
\node [fill=black] (u2) at (3, 5) {};
\node [fill=black] (u3) at (5, 1) {};
\node [] (v1) at (2, 4) {};
\node [] (v2) at (4, 6) {};
\node [] (v3) at (6, 2) {};
\draw [thick,-latex] (u1) -- (11, 3);
\draw [thick,-latex] (u2) -- (11, 5);
\draw [thick,-latex] (u3) -- (11, 1);
\draw [thick,-latex] (v1) -- (2, -4);
\draw [thick,-latex] (v2) -- (4, -4);
\draw [thick,-latex] (v3) -- (6, -4);

\node [fill=black] (b1) at (-2, 8) {};
\node [fill=black] (b2) at (-3, 0) {};
\node [fill=black] (b3) at (0, -3) {};
\node [fill=black] (b4) at (8, -2) {};
\draw [thick,-latex] (b1) -- (11, 8);
\draw [thick,-latex] (b2) -- (11, 0);
\draw [thick,-latex] (b3) -- (11, -3);
\draw [thick,-latex] (b4) -- (11, -2);

\node [] (w1) at (-1, 9) {};
\node [] (w2) at (7, 10) {};
\node [] (w3) at (10, 7) {};
\node [] (w4) at (9, -1) {};
\draw [thick,-latex] (w1) -- (-1, -4);
\draw [thick,-latex] (w2) -- (7, -4);
\draw [thick,-latex] (w3) -- (10, -4);
\draw [thick,-latex] (w4) -- (9, -4);

\draw [draw=white] (13.0, 3.0) node {$\implies$};

\begin{scope}[shift={(18.0, 2.0)}]
\draw [dashed,draw=black,fill=black!10!white] (0.5, -4) -- (0.5, 2.5) -- (7.0, 2.5) -- (7.0, 0.5) -- (2.5, 0.5) -- (2.5, -4) -- cycle;

\foreach \x in {-3, -2, -1, 0, 1, 2, 3, 4, 5, 6, 7} 
    \draw[very thin,dashed] (\x, -4) -- (\x, 6);
\foreach \y in {-4, -3, -2, -1, 0, 1, 2, 3, 4, 5, 6} 
    \draw[very thin,dashed] (-3, \y) -- (7, \y);

\tikzstyle{every node}=[draw,circle,fill=white,minimum size=4pt,inner sep=0pt]
\node [fill=black] (u) at (1, 1) {};
\node [fill=black] (b1) at (-2, 4) {};
\node [fill=black] (b2) at (-3, 0) {};
\node [fill=black] (b3) at (0, -3) {};
\node [fill=black] (b4) at (4, -2) {};
\node [] (v) at (2, 2) {};
\node [] (w1) at (-1, 5) {};
\node [] (w2) at (3, 6) {};
\node [] (w3) at (6, 3) {};
\node [] (w4) at (5, -1) {};
\draw [thick,-latex] (u) -- (7, 1);
\draw [thick,-latex] (b1) -- (7, 4);
\draw [thick,-latex] (b2) -- (7, 0);
\draw [thick,-latex] (b3) -- (7, -3);
\draw [thick,-latex] (b4) -- (7, -2);
\draw [thick,-latex] (v) -- (2, -4);
\draw [thick,-latex] (w1) -- (-1, -4);
\draw [thick,-latex] (w2) -- (3, -4);
\draw [thick,-latex] (w3) -- (6, -4);
\draw [thick,-latex] (w4) -- (5, -4);
\end{scope}

\end{tikzpicture}
\caption{Substitution of a buried subgraph in a {\twoDORG}.}
\label{fig:substitution}
\end{figure}
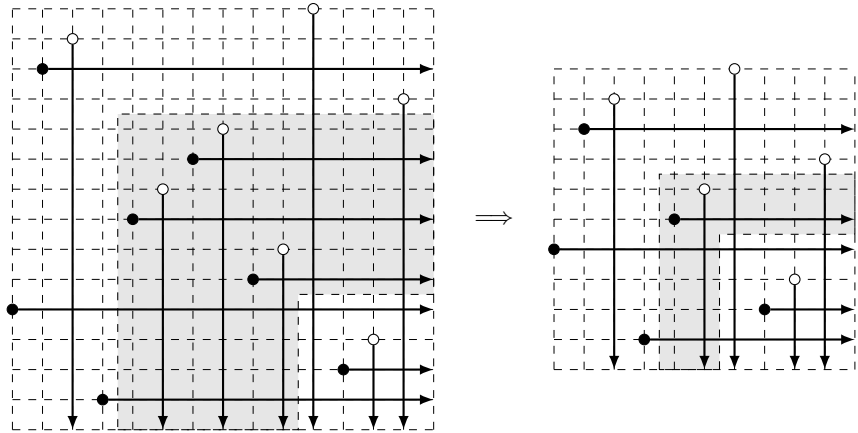

A future research direction is the decomposition of {\twoDORG}s.
The definition of buried subgraphs is reminiscent of the notion of bimodules 
introduced by Fouquet, Habib, Montgolfier, and Vanherpe~\cite{FHMV04-LNCS}.
They referred to a set of vertices as a \emph{bimodule} if it satisfies Condition (b) of Definition~\ref{def:buried subgraph}.
Decomposing {\twoDORG}s into uniquely representable graphs by bimodular decomposition or similar tools could be a promising research direction in the future. 
We hope our work would be a first step toward this decomposition framework. 

Another research direction is the investigation of the unique representability of other graphs related to interval graphs and {\twoDORG}s.
These graphs are known to share a common characterization with 
complements of threshold tolerance graphs (also known as co-TT graphs), 
adjusted interval digraphs, and 
signed-interval digraphs~\cite{HHMR20-SIDMA}.
We hope that studying unique representations of these graphs would be interesting for further research.

\subsubsection*{Acknowledgments.}
This work was supported by JSPS KAKENHI Grant Number JP23K03191. 

%
%
%
\bibliographystyle{splncs04}
\bibliography{ref.bib}
\end{document}